\documentclass[11pt,reqno]{amsart}
\usepackage{etex}
\usepackage{tikz}
\usepackage{tikz-cd}
\usepackage{dsfont}
\usepackage{fancyhdr}
\usetikzlibrary{matrix}
\usetikzlibrary{shapes}
\usepackage{pgfplots}
\usepgfplotslibrary{polar}
\pgfplotsset{compat=newest}
\usepackage{amssymb, amscd,amsmath}
\usepackage{graphicx}
\usepackage[all]{xy}
\usepackage{mathrsfs}
\usepackage{calc}
\usepackage{amsthm}
\usepackage{wrapfig}
\usepackage{float}
\usepackage{tabularx,ragged2e,booktabs,caption}
\usepackage{hyperref}

\newtheorem{theorem}{Theorem}[section]
\newtheorem{prop}[theorem]{Proposition}
\newtheorem{corollary}[theorem]{Corollary}
\newtheorem{lemma}[theorem]{Lemma}
\theoremstyle{remark}
\newtheorem{remark}[theorem]{Remark}
\theoremstyle{definition}

\newtheorem*{acknowledgements}{Acknowledgements}

\newcommand*\Co{\mathbb{C}}
\newcommand*\Z{\mathbb{Z}}
\newcommand{\Q}{\mathbb{Q}}

\newcommand{\F}{\mathbb{F}}

\newcommand{\as}{a_{\sigma}}
\newcommand{\bs}{b_{\sigma}}
\newcommand{\cs}{c_{\sigma}}
\newcommand{\ds}{d_{\sigma}}
\newcommand{\two}{\mathbb{Z}/2\mathbb{Z}}
\newcommand{\four}{\mathbb{Z}/4\mathbb{Z}}
\newcommand{\six}{\mathbb{Z}/6\mathbb{Z}}
\newcommand{\eight}{\mathbb{Z}/8\mathbb{Z}}
\newcommand{\sixteen}{\mathbb{Z}/16\mathbb{Z}}

\newcommand{\thirtytwo}{\mathbb{Z}/32\mathbb{Z}}
\newcommand{\GL}{\mathrm{GL}}
\newcommand{\SL}{\mathrm{SL}}

\DeclareMathOperator{\Aut}{Aut}
\DeclareMathOperator{\Gal}{Gal}
\DeclareMathOperator{\Hom}{Hom}

\begin{document}

\title[Torsion Subgroups of Elliptic Curves]{Torsion Subgroups of Elliptic Curves over Quadratic Cyclotomic Fields in Elementary Abelian $2$-Extensions}
\author{\"{O}zlem Ejder}
\address{Department of Mathematics\\
   University of Southern California\\
  Los Angeles, California 90089}
\email{ejder@usc.edu}
\date{April 2017}
\maketitle

  \begin{abstract}
     Let $K$ denote the quadratic field $\Q(\sqrt{d})$ where $d=-1$ or $-3$ and let $E$ be an elliptic curve defined over $K$. In this paper, we analyze the torsion subgroups of $E$ in the maximal elementary abelian $2$-extension of $K$.
  \end{abstract}

 \section{Introduction}
    Finding the set of rational points on a curve is one of the fundamental problems in number theory. Given a number field $K$ and an algebraic curve $C/K$,  the set, $C(K)$, of points on $C$ which are defined over $K$, has the following properties depending on the genus of the curve.
\begin{enumerate} 
  \item If $C$ has genus $0$, then $C(\Q)$ is either empty or infinite.
  \item If $C$ has genus greater than $1$, then $C(\Q)$ is either empty or finite. \\(Faltings's theorem)
\end{enumerate}
Assume $C$ has genus $1$. If $C(K)$ is not empty, then it forms a finitely generated abelian group, proven by Luis Mordell and Andr\'{e} Weil.  Thus, given an elliptic curve $E/K$, the group $E(K)$ has the structure
    \[ 
      E(K)_{\text{tors}} \oplus \Z^r.
    \]
Here $E(K)_{\text{tors}}$ is the finite part of this group $E(K)$, called the torsion subgroup and the integer $r$ is called the rank of $E$ over the number field $K$. 
  
In this paper, we will study the elliptic curves over the quadratic cyclotomic fields and how their torsion subgroups grow in the compositum of all quadratic extensions of the base field. 

First, we summarize the results obtained so far on $E(K)_{\text{tors}}$ for a number field $K$.  Mazur \cite{Mazur} showed that the only groups that can be realized as the torsion subgroups of  elliptic curves defined over $Q$ are the following:                             
      \begin{align*}   \label{Mazur}
                          \Z/m\Z \hspace{2mm} \text{for} \hspace{2mm}  1\leq m \leq 12,   m\neq 11,
                        \hspace{2mm} \text{or} \hspace{2mm}
                       \Z/2\Z \oplus \Z / 2m\Z  \hspace{2mm} \text{for} \hspace{2mm} 1\leq m \leq 4.
       \end{align*}
   Similarly, the list for the torsion groups of elliptic curves defined over a quadratic field has been given by S. Kamienny  \cite{Kamienny}, M.A. Kenku, and F. Momose \cite{Kenku-Momose}. 
           \[  
             \label{Kenku} \Z/m\Z  \hspace{2mm} \text{for} \hspace{2mm}1\leq m \leq 18, m\neq 17, 
                   \hspace{2mm}
        \Z/2\Z \oplus \Z/2m\Z    \hspace{2mm} \text{for} \hspace{2mm} 1\leq m \leq 6, 
                  \]
             \[ 
              \Z/3\Z \oplus \Z/3m\Z  \hspace{2mm} \text{for} \hspace{2mm} m=1,2,   \hspace{2mm} 
              \text{and} \hspace{2mm}
             \Z/4\Z \oplus \Z/4\Z.
              \]              
     If one fixes the quadratic field $K$, it is very likely that one will have a smaller list.  In fact, the groups $ \Z/3\Z \oplus \Z/3\Z$ and $ \Z/3\Z \oplus \Z/6\Z$ are only realized when $K=\Q(\sqrt{-3})$ where as the group  $ \Z/4\Z \oplus \Z/4\Z$ is only realized over the field $\Q(i)$ since they contain the roots of unity for $3$ and $4$ respectively (See Weil pairing). On the other hand, Filip Najman \cite{Najman-cyclotomic} has proved that 
        \begin{enumerate}  \label{Najman}
                  \item 
                          If $K=\Q(i)$, then $E(K)_{\text{tors}}$ is either one of the groups from Mazur's 
                            theorem or $\Z/4\Z \oplus \Z/4\Z$. 
                 \item
                            If $K=\Q(\sqrt{-3})$, then $E(K)_{\text{tors}}$ is either one of the groups from Mazur's 
                            theorem or $\Z/3\Z \oplus \Z/3\Z$ or $\Z/3\Z \oplus \Z/6\Z$.
        \end{enumerate}    
     One may also ask how the torsion subgroups of elliptic curves over a given number field $K$ grow over the compositum of all the quadratic extensions of $K$. 
     
     Let $F$ be the maximal  
     elementary abelian two extension of $K$, i.e.,                  
                     \[ 
                     F:=K[ \sqrt{d} : d \in \mathcal{O}_K ] 
                      \]    
  where $\mathcal{O}_K$ denotes the ring of integers of $K$. The problem of finding $E(F)_{\text{tors}}$ where $K=\Q$ has been studied by Michael Laska, Martin Lorenz \cite{Laska-Lorenz},  and Yasutsugu Fujita  
    \cite{Fujita1, Fujita2}. Laska and Lorenz described a list of $31$ possible groups and Fujita 
    proved that the list of $20$ different groups is complete.

    Our main theorem generalizes the results of Laska, Lorenz and Fujita to the 
    case 
    where $K$ is a quadratic cyclotomic field. We find that  (See Theorem \ref{main}.)

         \begin{enumerate}
                           \item If $K=\Q(i)$, then
                        $E(F)_{\text{tors}}$ is isomorphic to one of the following groups:
                        \begin{align*}
                          & \Z/{2}\Z \oplus \Z/{2N}\Z & &    ( N=2,3,4,5,6,8) & &\\
                          & \Z/{4}\Z \oplus \Z/{4N}\Z & &    ( N=2,3,4) & &\\
                          &\Z/N\Z \oplus \Z/N\Z && (N=2,3,4,6,8) & & 
                       \end{align*}
                 \noindent    or $\{1\}, \Z/3\Z$, $ \Z/5\Z$, $ \Z/7\Z$, $ \Z/9\Z$, $ \Z/15\Z$.                                  
                  \item If $K=\Q(\sqrt{-3})$, then $E(F)$ is either isomorphic to one of the groups listed above or 
                              \[ \Z/2\Z \oplus \Z/32\Z. \]             
               \end{enumerate}
 
 We first study the points on various modular curves and use these results in \S\ref{odd section} and in  \S\ref{restrictions} to prove Theorem \ref{Laska} which gives us a list of possible torsion subgroups. The main result of \S\ref{odd section} is Proposition \ref{odd} where we give the possible odd order subgroups of $E(F)$. \S\ref{Mordell} is concerned with finding non-trivial solutions to Fermat's quartic equation which is crucial to rule out the subgroup $\two \oplus \Z/32\Z$. All of these results together leads us to Theorem \ref{Laska}.
 
 In \S\ref{full torsion}, we analyze the growth of each non-cyclic torsion subgroup over the base field which helps us eliminate some of the groups given in Theorem \ref{Laska}. Analyzing the case where $E(K)$ is cyclic, we obtain more restrictions on the torsion subgroups in \S\ref{restrictions2} and we prove our main result Theorem \ref{main}.                        

\section{Background}
If $E$ is an elliptic curve given by the model 
   \[
    y^2=x^3+ax+b,
    \]
then the quadratic twist of $E$ by $d \in K$ is 
\[ 
E^{(d)}: y^2=x^3+ad^2x+bd^3 
\]
and it is isomorphic to $dy^2=x^3+ax+b$. Note that if $d$ is a square in $K$, then $E^{(d)}$ and $E$ are isomorphic over $K$.

To compute the rank of elliptic curves over quadratic fields, we make use of the following Lemma.

 \begin{lemma}[{\cite[Corollary 1.3]{Laska-Lorenz}}] \label{rank}
 Let $d$ be a square-free integer. Then for an elliptic curve $E/\Q$, the following holds:
  \[
   \text{rank}(E(\Q(\sqrt{d}))) = \text{rank}(E(\Q)) + \text{rank}(E^{(d)}(\Q)).
   \]
\end{lemma}
 Throughout the paper, we compute the rank of $E(\Q)$ and $E^{(d)}(\Q)$ on Magma \cite{Magma} and use Lemma \ref{rank} to compute the rank of $E(\Q(\sqrt{d}))$. The torsion subgroup of a given elliptic curve is also computed using Magma. 
 
 Let $K$ be a number field. Given an elliptic curve $E/K$, we call a subgroup $C$ of $E(\mathbb{C})$ as
    $K$-rational if there exists an elliptic curve $E'/K$ and an isogeny 
           \[ 
               \phi: E \to E'
           \] 
  defined over $K$ such that $C=\ker{\phi}$. Equivalently, $C$ is $K$-rational if it is invariant under the action of $\Gal(\bar{K}/K)$.
  
   For our later purposes, we will state the following result of Newman, which tells us about the existence of $K$-rational subgroups of certain degrees in quadratic extensions of $K$. 
   
   Let $K=\Q(\sqrt{d})$ for $d=-1,-3$.
        
     \begin{theorem}[{\cite[Theorem~8]{Burton-odd}}] \label{modular curves-Burton}
                         Let $E/K$ be an elliptic curve. Then $E(\Co)$ has no $K$-rational cyclic subgroups of order   
                        $24,35$ or $45$ defined over a quadratic extension of $K$. Moreover, if $E$ is defined 
                        over $\Q(\sqrt{-3})$, $E$ does not 
                        have a $K$-rational cyclic subgroup of order $20, 21$ or $63$ defined over a quadratic 
                        extension of $\Q(\sqrt{-3})$.
       \end{theorem}
       
       \begin{remark}\label{explain Burton}
       The proof of Theorem~8 in \cite{Burton-odd} shows that the modular curve $X_0(n)(K)$ has 
       no non-cuspidal points for $n=24,35$ or $45$. Moreover, if $K=\Q(\sqrt{-3})$, then $X_0(20)(K)$ 
       also does not have any non-cuspidal points. 
       \end{remark}
       
   We will use the following result to determine the odd torsion subgroup in an elementary abelian  
      $2$-extension of a field. Remember that the multiplication by $n$ map on an elliptic curve is denoted by $[n]$.
        \begin{lemma}[{\cite[Corollary 1.3]{Laska-Lorenz}}] 
                     \label{Laska-Lorenz}
             Let $E$ be an elliptic curve over the field $k$ and let $L$ be an elementary abelian 
             $2$-extension of $k$ of degree $2^m$, i.e., the Galois group of $L/k$ is an elementary 
             abelian $2$-group of rank $m$. If 
             $E(k)_{2'}=\{ P \in E(k) \hspace{0.7mm} : [n]P=0 \hspace{1mm} \text{for  some odd} 
             \hspace{0.8mm} n\}$, 
             then
                   \[ 
                   E(L)_{2'} \simeq E^{(d_1)}(k)_{2'} \oplus \hdots \oplus E^{(d_m)}(k)_{2'},
                  \]
              for suitable $d_i : i=1,\ldots,m$ in $\mathcal{O}_k$. Furthermore, the image of each   
             summand $E^{(d)}(k)_{2'}$ is a $k$-rational subgroup of $E(L)$.
       \end{lemma}

 In \S\ref{restrictions}, we will use the correspondence between the lattices in $\Co$ and the elliptic curves over $\Co$ to describe some explicit isogenies.
 
 Now, let $L_1$ and $L_2$ be lattices inside $\Co$ such that $\alpha L_1 \subset L_2$ for some $\alpha$ in $\Co$. 
 Then the map $\Co/L_1 \to \Co/L_2$ induced by 
 \[ 
 z \mapsto \alpha z
 \]
  defines an isogeny in $\Hom(E_1,E_2)$ where $E_1,E_2$ are the elliptic curves corresponding to the lattices $L_1,L_2$, as in \cite[VI, Proposition~3.6b]{Silverman}. We will denote this isogeny by 
$[\alpha]_{E_1,E_2}$. When $E_1=E_2$, we will use $[\alpha]_{E_1}$ and we drop the domain and the target when it is clear from the context.

\section{Elementary Abelian $2$-Extensions}
For the rest of the paper, let $K$ be the quadratic field $\Q(\sqrt{d})$ for $d=-1, -3$ and let $F$ be the field
   \[
   K(\sqrt{d} : d \in \mathcal{O}_K).
   \]
  The field $F$ is called the maximal elementary abelian $2$-extension of $K$ since its Galois group is an elementary abelian $2$-group and it is maximal with respect to this property. Let $E/K$ be an elliptic curve given by $y^2=f(x)$. We make the following quick observations.
\begin{enumerate}
\item  If $f$ is irreducible in $K$, then it remains irreducible over the field $F$. Otherwise, $f$ has a 
        root $\alpha$ in $F$ and the degree of $K(\alpha)$ over $K$ is divisible by $3$ but it is not 
        possible since $K(\alpha)$ is contained in $F$.
        
 \item If $E(K)$ does not have a point of order $2$, then $E(F)$ cannot have a point of order $2$ 
      either. This simply follows from the fact that the points of order $2$ on the elliptic curve $E$ are  
      given by the zeros of 
      $f$, i.e., 
      \[  
      \{(\alpha,0) : \hspace{.5mm} f(\alpha)=0\}.
      \]
       Therefore, if $E(K)$ does not have a point of  
      order $2$, then $f$ is irreducible over $K$ and the claim follows from the first observation.

 \item  If the $j$-invariant of $E$ is not $0$ or $1728$, then for any elliptic curve 
       $E'/K$ isomorphic to $E$, we have $E(F) \simeq E'(F)$ since $E$ and $E'$ are isomorphic over a 
       quadratic extension of $K$ (hence also over $F$). 
       
\item  Let $L$ be a finite elementary abelian extension of $K$. Then for a prime $\mathfrak{P} \in 
       \mathcal{O}_L$ above $\mathfrak{p}$ in $\mathcal{O}_K$, the residue field of $\mathcal{O}_L$ is 
       at most a quadratic extension of the residue field of $\mathcal{O}_K/\mathfrak{p}$ since the 
       Galois group of the finite field $\F_{q^n}$ over $\F_q$ is cyclic for every $n$ and $q$. 
 
   \end{enumerate}          
         

\section{Rational Points on the Modular Curve $Y_0(n)$}

     \begin{prop} \label{modular curves}
                   Let $E/K$ be an elliptic curve. Then $E(\Co)$ has no $K$-rational cyclic subgroups of 
                   order  
                           $32, 36$ or $64$.
      \end{prop}
      \begin{proof} 
             An affine model for the modular curve $X_0(32)$ is given in \cite[p~503]{Yang} as
                              \[ 
                               y^2=x^3+4x.
                               \]
             Ogg's theorem \cite{Ogg} tells us that $X_0(32)$ has the following cusps.                       
                    \[
                           \begin{tabular}{  m{1.8cm} | m{1cm}| m{1cm} | m{1cm} | m{1cm}| m{1cm} | m{1cm} } 

                                       d & 1 &2 & 4& 8 &16 &32 \\   \hline
                                      $ \phi(d, 32/d)$ & 1 & 1 & 2 &2 &1 &1 \\ 
                           \end{tabular}
                       \]  
              We see that $X_0(32)$ has $8$ cusps; four of them defined in $\Q$ and the other four
              defined in $\Q(i)$. We compute on Magma that $X_0(32)$ has only $8$ points over $\Q(i)$ and has $4$ points  over $\Q(\sqrt{-3})$, hence they are all cuspidal.       
            
              This shows that there are no elliptic curves $E/K$ with a $K$-rational cyclic subgroup of order 
              $32$.
              
              Similarly,  an affine model for the modular curve $X_0(36)$ is also given in \cite[p~503]{Yang} as
                    \[
                                    y^2=x^3+1
                    \]
              We again apply Ogg's method to compute the cusps of the modular curve $X_0(36)$.                  \[
                            \begin{tabular}{  m{1.8cm} | m{.7cm}| m{.7cm} | m{.7cm} | m{.7cm}| m{.7cm} | 
                            m{.7cm} | m{.7cm} | m{.7cm} | m{.7cm} } 
                                            d & 1 & 2 & 3& 4 &6 &9 & 12& 18 & 36 \\ \hline
                                   $\phi(d, 36/d)$ & 1 & 1 & 2 &1 &2 &1 &2 & 1 &1  \\ 
                           \end{tabular}
                    \]  
              The table shows that $X_0(36)$ has $12$ cusps; six of them defined over $\Q$ and the 
              remaining six
              are defined over $\Q(\sqrt{-3})$. We find that it has $6$ points over $\Q(i)$ and  $12$ points over $\Q(\sqrt{-3})$.                          
          
             This shows that there are no non-cuspidal $K$-points on $X_0(36)$ and so
             there are no elliptic curves $E/K$ with a $K$-rational cyclic subgroup of order $36$.     
           
          Finally, $E(\Co)$ has no $K$-rational cyclic subgroup of order $64$ since otherwise it would induce a 
           $K$-rational cyclic subgroup of order $32$.
      \end{proof}  
   
   Now, we will study the $K$-rational points on $X_0(20)$ and $X_0(27)$ to prove Proposition \ref{odd} and also to prove the results of \S\ref{restrictions} later.
 \subsection{The modular curve $X_0(20)$.} \label{modular curve 20}
     
      Let $K=\Q(i)$.
      An equation for the modular curve $X_0(20)$ is given in \cite{Yang} as
           \[ 
           y^2=(x+1)(x^2+4).
           \] 
        It is known that there are no cyclic $20$-isogenies defined over $\Q$, see Theorem~2.1 in \cite{Laska-Lorenz}, hence $X_0(20)(\Q)$ has only 
        cusps. Oggs's method tells us that there are only $6$ of them. Then we compute on Magma that $X_0(20)(K)$ has $12$ points and they are listed as
                                    \[
                       \{\mathcal{O}, (-1,0), (0,\pm 2), (4,\pm 10), ( (\pm2i,0),(2
                       i-2,\pm(2i+4)),(-2i-2,\pm(2i-4))\}. 
                   \]
       This shows that there are $6$ 
            non-cuspidal points on $X_0(20)(K)$. We will study these points in more detail in Proposition \ref{20}.
            
 \subsection{The modular curve $X_0(27)$.}  \label{27}
       A model for the modular curve $X_0(27)$ is given in \cite{Yang} as
                          \[  
                          y^2+y =x^3-7.    
                           \]
       Again by Ogg's method, we find that 
       $X_0(27)$ has $6$ cusps; four of them defined over $\Q(\sqrt{-3})$ and the other two defined over 
       $\Q$. 
       
     Now let $K=\Q(\sqrt{-3})$. We compute that $E(K)$ has $9$ points and hence there are $3$ non-cuspidal points on $X_0(27)$ defined over $K$.
       
  Similarly if $K=\Q(i)$, the group $E(K)$ is also finite and it has $3$ points which shows that there is only one non-cuspidal point on $X_0(27)$ defined over $\Q(i)$, in 
        fact defined over $\Q$.
       
       Let $E_1$ be the elliptic curve associated with the lattice $[1, \frac{1+\sqrt{-27}}{2}]$. Then
                    \[
                   [ \sqrt{-27}]: E_1\to E_1  
                    \] 
            and        
                    \[
                    [ \frac{9\pm \sqrt{-27}}{2}]: E_1 \to E_1 
                      \]
           define endomorphisms of $E_1$ and they are cyclic of degree $27$. Moreover $E_1$ has complex multiplication by the order $\Z[\frac{1+\sqrt{-27}}{2}]$ and it is given in 
           \cite[p.261]{Cox} that
                  \[ 
                  j(E_1)=-2^{15}5^33. 
                  \]    
Therefore the endomorphisms listed above are defined over $\Q(\sqrt{-3})$ by \cite[Theorem~2.2]{silverman-advance}.  We find a model for $E_1$ in the database \cite{LMFDB}; the elliptic curves over $\Q$ with 
  complex multiplication, as
  \[ 
  y^2+y=x^3-270x-1708.
  \]
  Hence any elliptic curve defined over $K$ with a $K$-rational cyclic subgroup of order $27$ is a quadratic twist of $E_1$.

\section{Odd Torsion} \label{odd section}
 \noindent Using Lemma \ref{Laska-Lorenz}, we see that the odd primes dividing the order of a point in $E(F)$ can only be $3,5$ or $7$. We will prove in Proposition \ref{21-2} and Proposition \ref{21} that $E(F)$ does not have a point of order $21$ for $K=\Q(\sqrt{d})$ for $d=-1,-3$.

       \begin{prop} \label{21-2} Let $E$ be an elliptic curve defined over $K=\Q(\sqrt{-3})$. Then $E(F)$ has no point of order $21$.
       \end{prop}
       
       \begin{proof} Assume $E(F)$ has a subgroup of order $21$. Then by Lemma \ref{Laska-Lorenz}, 
       replacing $E$ by a twist if necessary, we may assume that $E(K)$ has a point of order $3$ and 
       $E^{(d)}(K)$ has a point of order $7$ for 
       some $d$ in $\mathcal{O}_K$, hence $E$ has a subgroup of order $21$ over a quadratic extension 
       of $K$ and it is $K$-rational by Lemma \ref{Laska-Lorenz}. Theorem \ref{modular curves-Burton} shows that it is not possible.
       \end{proof}
       
       \begin{remark} 
              The modular curve $X_0(21)$ is an elliptic curve with Mordell-Weil rank $1$ over $\Q(i)$. 
           Hence $X_0(21)$ can not be immediately used to determine whether an elliptic curve $E$ can 
           have a subgroup of order $21$ defined over the field $F$. 
       \end{remark}
    
    We will need the following result in the proof of Proposition \ref{21}.
     
     \begin{theorem}[{\cite[Theorem7]{Burton-odd}}] \label{Burt-0,1728}
                         Let $K$ be a quadratic field and let $E/K$ be an 
                         elliptic curve. If $j(E) = 0$ and $p > 3$ is a prime, then $E(K)_{\text{tor}}$ has no 
                         element of order $p$. If $j(E) = 1728$ and $p > 2$ is a prime, then $E(K)_{\text{tor}}$ 
                         has no element of order p.
       \end{theorem}
       \begin{proof} The proof uses the techniques 
                                from \cite{Lemmermeyer}.           
       \end{proof}

             \begin{prop}\label{21}
               Let $E$ be an elliptic curve defined over $\Q(i)$. Then $E(F)$ does not have a subgroup of 
               order $21$.
       \end{prop}
       \begin{proof} Let $E$ be an elliptic curve defined over $K=\Q(i)$ and suppose that $E(F)$ has a 
                subgroup of order $21$. We may assume that $E(K)$ has a  point of order $7$ (by replacing 
                with a twist if necessary) by Lemma \ref{Laska-Lorenz}. It can be found in 
                \cite[Table 3, p~217]{Kubert} that an elliptic curve with a point of order $7$ is isomorphic 
                to
                       \begin{equation} \label{parametrization}
                                 E_t : y^2+(1-c)xy-by=x^3-bx^2
                        \end{equation} 
               where $b=t^3-t^2$ and $c=t^2-t$ for some $t\neq 0,1$ in $K$. Therefore $E$ is isomorphic 
               to $E_t$ for some $t \in K$. Moreover, either the j invariant of $E$ 
             is $0$ or $1728$, or the isomorphism is defined over a quadratic extension of $K$. By 
             \cite[Theorem 7]{Burton-odd}, we know that an elliptic curve with j invariant $0$ or $1728$ can 
             not have a $K$-point of order $7$, hence $E$ is a quadratic twist of $E_t$ and if 
             $E(F)$ has a point of order $21$, then so has $E_t$; hence we may assume
             $E$ is $E_t$.
 
            We compute the third division polynomial of the elliptic curve $E_t$ as the following.
                           \begin{multline*}\psi(x,t)= x^4 + \left(\frac{1}{3}t^4 - 2t^3 + t^2 +     \frac{2}{3}t + 
                                                                     \frac{1}{3}\right)x^3  + \left( t^5 - 2t^4 + t^2 \right)x^2 \\ +  
                                                                     \left(t^6 - 2t^5 + t^4\right)x + \left(- \frac{1}{3}t^9 + t^8  
                                                                     - t^7 + \frac{1}{3}t^6\right). 
                            \end{multline*}
            Now $E$ has a point $P$ of order $3$ defined in a quadratic extension of $K$ and the 
            subgroup generated by $P$ is $K$-rational by Lemma \ref{Laska-Lorenz}. We claim that $x(P)
            $, the $x$-coordinate of the point $P$ must be in $K$ which forces the equation $\psi(x,t))=0$ 
            to have a root in $K$. The only points in $\langle P \rangle$ are $P$, $-P$ and the point at 
            infinity. Since $x(P)=x(-P)$, it follows that $\sigma(x(P))=x(\sigma(P))=x(P)$ and hence $x(P)$ 
            is invariant under the action of the Galois group.
            Now, the pair $(E,P)$ corresponds to a point $(x_0,s)$ on the curve $C$ given by the equation         
                         \[
                          C: \psi(x,t)=0 
                          \] 
            where $E=E_s$ and $x_0$ is the $x-$coordinate of the point $P$ of order $3$. 
            Therefore it is enough to find $C(K)$, the set of $K$-points on $C$. The curve $C$ is birational 
            (over $\Q$) to the hyperelliptic curve  
                    \[ 
                    \tilde{C}: y^2=f(u)
                    \]  
                       where 
                                 \[
                                  f(u)=u^8-6u^6+4u^5+11u^4-24u^3+22u^2-8u+1
                                  \]
            and $C$ and $\tilde{C}$ are isomorphic over $K$ outside the set of singularities which is $\{(0,0),(0,1) \}$. Note that we require $t$ to be different than $0$ or $1$ in 
            (\ref{parametrization}), hence it is enough to find $\tilde{C}(K)$. The polynomial $f(u)$ factors 
            as 
                        \[ 
                        f(u)=(u^2 - u + 1)(u^6 + u^5 - 6u^4 - 3u^3 + 14u^2 - 7u + 1)
                         \]
            Let $g$ and $h$ denote the factors of $f$:
                         \[ 
                         g(u)=u^2 - u + 1 
                         \] 
                     \[
                      h(u)=u^6 + u^5 - 6u^4 - 3u^3 + 14u^2 - 7u + 1.
                      \] 
            Using the Descent Theorem (\cite[Theorem 11]{Stoll}; one can also look at Example $9$ and 
            $10$ in \cite{Stoll}.), it is enough to find the points on the unramified coverings $\tilde{C}_d$ of 
            $\tilde{C}$, which are given as the intersection of two equations in $\mathbb{A}^3$:
                           \[  
                           w^2=d g(u)=d(u^2 - u + 1) 
                           \] 
            and  
                    \[ 
                    z^2=d h(u)=d( u^6 + u^5 - 6u^4 - 3u^3 + 14u^2 - 7u + 1) 
                    \] 
            where $d$ is a square-free number in $\mathcal{O}_K$ dividing the resultant of 
            $g(u)$ and  
            $h(u)$, which is $112$. Therefore $d$ belongs to the set
                      \[ 
                      \{1,i, (1+i),7,i(1+i),7(1+i),7i, 7i(1+i)\}
                       \] 
            If we exclude the cases $d=1$ and $d=7i$, reduction of $d$ takes values $\{2,3,2,1,1,2\}$ and 
            $\{3,4,2,2,3,4\}$ with respect to the ideal $(2-i)$ and $(2+i)$ respectively. We will reduce the 
            curve                   
                     \[
                     z^2=d(u^6 + u^5 - 6u^4 - 3u^3 + 14u^2 - 7u + 1)  
                     \] 
            at $(2-i)$ for the values of $d=i,(1+i),7,7i(1+i)$ and similarly reduce it at $(2+i)$ for 
             the values 
             $d=i(1+i),7(1+i)$.
            In each case described above, $z^2=d h(u)$ reduces to either
                     \[
                     z^2=2(u^6 + u^5 - 6u^4 - 3u^3 + 14u^2 - 7u + 1).
                      \] 
                 or    
                  \[
                     z^2=3(u^6 + u^5 - 6u^4 - 3u^3 + 14u^2 - 7u + 1).
                      \]   
            A quick computation on Magma shows that neither of these equations has a
            solution over 
            $\F_5$, hence there are no $K$-points on $\tilde{C}_d$ for $d\neq 1,7i$. 
           
           Let $d=7i$. Magma computes $0$ as an upper bound for the Mordell-Weil rank of the Jacobian 
           of the curve 
                    \[
                    z^2=7i\left(u^6 + u^5 - 6u^4 - 3u^3 + 14u^2 - 7u + 1\right), 
                    \]  
            hence the rank of the Jacobian of $z^2=7i h(u)$ is zero. Moreover, we compute on 
            Magma that the Jacobian of the hyperelliptic curve $z^2=(7i) h(u)$ has $79$ and $171$ points 
            respectively over the finite fields $\F_5$ and $\F_{13}$ (reduced at $(2-i)$ 
            and $(2-3i)$) which proves that the torsion subgroup of the Jacobian of $z^2=(7i) h(u)$ over $K$ is trivial.
            Therefore $\tilde{C}_d(K)=\emptyset$ for $d=7i$.
            Hence, if there is a point on the curve $\tilde{C}(K)$, it must arise from the covering 
            $\tilde{C}_1(K)$. 
            
            Now we may assume that $d=1$. We would like to find the 
            $K$-points on the curve 
                     \[ 
                     C_2: z^2=h(u)=u^6 + u^5 - 6u^4 - 3u^3 + 14u^2 - 7u + 1.
                     \] 
            Magma computes that the rank of $J(C_2)(\Q)$ is $2$ and also that $2$ is an an 
            upper bound for $J(C_2)(K)$, therefore the rank of $J(C_2)$ over $K$ is equal to its rank over $\Q$.
            
            Similar to the case $d=7i$, the reduction of $C_2$ at the good primes $(2-i)$ and $(2-3i)$ has $79$ and $171$ points over $\F_5$ and $\F_{13}$ respectively, hence $J(C_2)(K)_{\text{tors}}$ is also trivial.
                        
            Let $J$ denote $J(C_2)$. We claim that $J(K)=J(\Q)$. Since $J(K)$ has rank $2$, we can find 
            $x,y \in J(K)$ such that $J(K)$ is generated by $x$ and $y$ as an abelian 
            group.  Let $\sigma$ be the generator of the $\Gal(K/\Q)$ and assume that $\sigma(x)=ax+by$ 
            for some $a,b \in \Z$. 
            
            Let $x'=nx+my$ and $y'=rx+sy$ be the generators of $J(\Q)$. Then 
            $sx'-my'$ is not zero since 
            $x',y'$ are the generators of a free abelian group. So a multiple 
            of $x$ (namely $(sn-mr)x$) is in 
            $J(\Q)$ and it is fixed by $\sigma$. Then we obtain $lax+lby=lx$ 
            (we use $l$ for $sn-mr$ above to simplify the notation). Since $J(K)$ has trivial torsion, it 
            implies that $(a-1)x+by=0$. We conclude that $a=1$ and $b=0$, i.e., $x$ in $J(\Q)$. A similar argument shows that  
            $y$ is also in $J(\Q)$, consequently $J(K)=J(\Q)$.
            
            We claim that $C_2(\Q)=C_2(K)$. 
            Let $P$ be a point in $C_2(K)$. If $P_0$ denotes the point $[0:1:1]$ on $C_2$, then $[P-P_0]$ 
            represents a point in $J(K)$ which equals to $J(\Q)$. If $P'$ denotes the Galois conjugate of 
            $P$, then 
                $[P'-P_0]$ must be equal to $[P-P_0]$, since a point in $J(\Q)$ is invariant under the action of $\Gal(\bar{\Q}/\Q)$. Hence
             $P=P'$ and $P$ is in $C(\Q)$ which proves that $C_2(\Q)=C_2(K)$. 
            
            Now we will show  
            that  $\tilde{C}(\Q)=\tilde{C}(\Q(i))$. 
            Remember that the curve $\tilde{C}$ is given by              
                \[ 
                y^2=u^8-6u^6+4u^5+11u^4-24u^3+22u^2-8u+1
                \]
            and we showed by the Descent theorem that a point $(u,y)$ in $\tilde{C}(K)$ 
            corresponds to a point on the intersection of
                                   \[ 
                                   C_1: w^2=u^2 - u + 1 
                                   \] 
              and  
                               \[ 
                               C_2: z^2=u^6 + u^5 - 6u^4 - 3u^3 + 14u^2 - 7u + 1 
                               \] 
             such that $y=wz$ where $(u,w) \in C_1(K)$ and $(u,z)\in C_2(K)$.
            The first equation         
                       \[ 
                       w^2=u^2-u+1=(u-1/2)^2+3/4 
                       \] 
            implies that if $(u,w)$ is a point on $C_1(K)$ with $u$ in $\Q$, then $w$ is also in $\Q$. 
            Hence, we showed that if $(u,y)$ is a point on $\tilde{C}(K)$, then $u,z$ are both in $\Q$ 
            since 
            $C_2(K)=C_2(\Q)$ and by the above argument, $w$ is also in $\Q$. Therefore $y=w z$ is also 
            in $\Q$.

            To summarize, we proved our claim that $\tilde{C}(K)=\tilde{C}(\Q)$. 
            This implies that a pair $(E,P)$ on $C(K)$ corresponds to a point $(u,y)$ on $\tilde{C}(\Q)$ and 
            therefore to a point in $C(\Q)$. However, if $E$ is defined over $\Q$ and $x(P) \in \Q$, then $P
            $ is in $E(\Q(\sqrt{d}))$ for some $d \in \Q$. We know by \cite{Laska-Lorenz} that $E$ does not 
            have a subgroup of order $21$ defined over a quadratic extension of $\Q$. Therefore there is 
            no elliptic curve $E$ defined over $K=\Q(i)$ such that $\Z/21\Z \subset E(F)$.

      \end{proof}

     \begin{prop}\label{odd}
            Let $K$ be a quadratic cyclotomic field and let $E$ be an elliptic curve defined over $K$. Then 
            $E(F)_{2'}$ is isomorphic to one of the following groups:
                   \[ 
                   \Z/N\Z \hspace{2mm} \text{for}  \hspace{2mm} N \in      \{1,3,5,7,9,15 \} \hspace{2mm} 
                         \text{or} \hspace{2mm} \Z/3\Z \oplus \Z/3\Z. 
                   \]
     \end{prop}
      \begin{proof}
          By Lemma \ref{Laska-Lorenz} and \cite[Theorem~2]{Najman-cyclotomic}, we see that the odd numbers dividing 
          the order of $E(F)_{\text{tors}}$ are products of  $3,5,7$ and $9$. Since $F$ does not contain a 
          primitive $n$-th root of unity for $n=5,7$ or $9$, $\Z/5\Z \oplus \Z/5\Z$, $\Z/7\Z \oplus \Z/7\Z$ or 
          $\Z/9\Z \oplus \Z/9\Z$ can not be isomorphic to a subgroup of $E(F)$ by the Weil pairing. 
          
          If $\Z/3\Z \oplus \Z/7\Z$ or $\Z/9\Z \oplus \Z/7\Z$ is a subgroup of $E(F)$, then it is Galois 
          invariant by Lemma \ref{Laska-Lorenz}, hence by Proposition \ref{21-2} and 
          Proposition \ref{21}, this is not possible. (Note that if $E$ has a Galois invariant subgroup of 
          order $63$, then it also has a Galois invariant subgroup of order $21$ which is not possible for 
          $\Q(i)$ by Proposition \ref{21} and for $\Q(\sqrt{-3})$ by Proposition \ref{21-2}.) 
          Similarly, $\Z/5\Z \oplus \Z/7\Z$ and  $\Z/5\Z \oplus \Z/9\Z$ are also not possible by Proposition 
          \ref{modular curves-Burton}. (See also Remark \ref{explain Burton} following 
          Proposition \ref{modular curves-Burton}).

          If $E(F)$ contains a subgroup isomorphic to $\Z/3\Z \oplus \Z/9\Z$, then by 
          Lemma \ref{Laska-Lorenz}, we may assume that $E(K)$ has a point $P$ of order $9$ and has 
          an additional $K$-rational subgroup $C$ of order $3$ arising from a twist, in other words, the subgroup $C$ is defined in a quadratic extension $K(\sqrt{d})$ of $K$. We will show that $E[3]$ is a subset of $E(K)$ and obtain a contradiction since $E(K)$ can not have a subgroup of order $27$ by \cite[Theorem~2]{Najman-cyclotomic}.
          
          Let $\sigma$ be the generator of the Galois group of $K(\sqrt{d})$ over $K$. Then the image of $\sigma$ under the map (Galois action on the set of $3$-torsion points)
           \[ 
            \Gal(K(\sqrt{d})/K) \rightarrow \GL_2(\F_3)
           \]
    is  $\begin{bmatrix} 1 & \alpha \\ 0 & \beta \end{bmatrix}$ for some $\alpha,\beta$ in $\F_3$. Note here that $\beta$ is $1$ modulo $3$ since $K$ contains a third root of unity $\zeta_3$,  $\beta$ is the determinant of the matrix of $\sigma$ and $\sigma(\zeta_3)=\zeta_3^{\beta}$. The fact that $\sigma$ has order $2$ tells us that $\alpha=0$ and $E[3]$ is a subset of $E(K)$ as we claimed.
    
Hence $E(F)$ can not have a subgroup isomorphic to  $\Z/3\Z \oplus \Z/9\Z$.                    
         
 Finally, $\Z/3\Z \oplus \Z/3\Z \oplus \Z/5\Z$ can not be isomorphic to a subgroup of $E(F)$ 
          either. Otherwise, $E(F)$ has a $K$-rational subgroup of order $15$ and by 
          \cite[Lemma 7]{Najman-cubic}, $E$ is isogenous to an elliptic curve with a cyclic $K$-rational subgroup of order $45$ contradicting Proposition \ref{modular curves-Burton}.
          See \cite[Proposition 2.2]{Laska-Lorenz} for the case $K=\Q$.
     \end{proof}
\section{Solutions to the Equation $x^4+y^4=1$.}\label{Mordell}
In this section, we will find points on the affine curve $x^4+y^4=1$ which is an affine model for $Y_0(64)$ \cite[Proposition~2]{Kenku64}. The results of Theorem~\ref{64} will be used in Proposition \ref{j}.

 We will call a solution $(x,y)$ (resp. $(x,y,z)$) of a Diophantine equation trivial if $xy=0$ (resp. $xyz=0$) and non-trivial otherwise.

   \begin{lemma} \label{diophantine}Let $K=\Q(\sqrt{d})$ for $d=-1$ or $-3$.
   \begin{enumerate}
     \item
        The only solutions of $x^4+y^2=1$ defined over $K$ are trivial.
     \item \label{difference} 
        The only solutions of $x^4-y^4=z^2$ defined over $K$ are trivial.
     \item 
        The only solutions of  $x^4+y^4=z^2$ defined over $\Q(i)$ are trivial.

    \end{enumerate}
   \end{lemma}

        \begin{proof} \leavevmode

           \begin{enumerate}
               \item
                   We can find a rational map between 
                     $C: x^4+y^2=1$  and $E:y^2=x^3+4x$ such that
                             \[ 
                             f: (x,y) \mapsto \left(\frac{2x^2}{1-y}, \frac{4x}{1-y}\right). 
                             \]
                     Since the $K$-valued points on $C$ maps to $K$-valued points on $E$, it is 
                     enough to 
                     look at the inverse images of $K$-points on $E$ and the points where the map $f$ is not 
                    defined.
                    We previously computed in the proof of Proposition \ref{modular curves} that
                       \[ 
                         E(\Q(\sqrt{-3}))=\{(0,0), (2,\pm 4) \} 
                       \]
                    and   
                      \[ 
                          E(\Q(i))=\{ (0,0), (2,\pm 4), (\pm 2i,0), (-2, \pm 4i)\} 
                      \]
                   We easily compute that the inverse image of $E(\Q(i))$ under $f$ is the set
                     \[ 
                        \{ (0,- 1),(\pm i,0),(\pm 1,0). \}
                      \]
                  The points where $f$ is not defined are the points where $y=1$ and so we obtain the point
                   $ (0,1)$. All of the solutions we found are trivial.
              \item 
              Let $(a,b,c)$  be a solution to $x^4-y^4=z^2$ over $K$. Assume $a^2\neq c$, then 
                    $(\frac{2b^2}{a^2-c},\frac{4 a b}{a^2-c})$ is a point on $E: y^2=x^3+4x$. As we did earlier, we simply compute the points $(a,b,c)$ for each point in 
                       $E(K)$ (which is described in the first part) and the points $(a,b,c)$ 
                   with $a^2=c$  to show that $a b c=0$. 
             \item 
                       Let 
                             \[
                              f: \left(x^4+y^4=z^2\right) \longrightarrow \left(y^2=x^3-4x\right)
                               \]
                             \[ 
                             (x,y,z) \mapsto \left( \frac{-2x^2}{y^2-z}, \frac{4xy}{y^2-z}\right)
                              \]
                      A quick computation on Magma shows that the affine curve $y^2=x^3-4x$ has only $3$ points 
                      defined over $\Q(i)$ and they are 
                      $(0,0)$ and $ (\pm 2,0)$. Hence if 
                      $(a,b,c)$ is a solution to $x^4+y^4=z^2$ over $\Q(i)$, then either $a=0$ or $b=0$. 
                      Notice 
                      also that if $f$ is not defined at $(a,b,c)$, then $a=0$. Hence the only solutions defined 
                      over 
                     $\Q(i)$ are trivial.    
           \end{enumerate}
        \end{proof}
        \begin{theorem} \label{64}
          Let $K=\Q(i)$. Assume that $x^4+y^4=1$ has a solution in a quadratic  
            extension $L$ of $K$. Then $L=K(\sqrt{-7})$ and the only solutions are 
               \[
                (\epsilon_1 \frac{1+\epsilon_3 \sqrt{-7}}{2}, \epsilon_2 \frac{1-\epsilon_3 \sqrt{-7}}{2} ), 
                \]
                \[(\pm i ,0), \hspace{1mm} (\pm 1,0), \hspace{1mm} (0, \pm i), \hspace{1mm} (0,\pm 1)   
                \]
        where
            \[ 
            \epsilon_{1,2}=\pm i \hspace{1mm}\text{or} \hspace{1mm} \pm 1,  \hspace{1mm} \epsilon_3=\pm1, \hspace{2mm} \text{ and}  
            \hspace{2mm} i=\sqrt{-1}. 
            \]
   \end{theorem}
   \begin{proof}
      Mordell (Chapter 14, Theorem 4 of \cite{Mordell}) proves that  if $x^4+y^4=1$ has a nontrivial 
      solution in a quadratic extension $L$ of $\Q$, then $L$ is $\Q(\sqrt{-7})$. We will use his technique 
      to 
      show that this result still holds if one replaces $\Q$ with $\Q(i)$.
      Let $L$ be a quadratic extension of $K$ and let $(a,b)$ be a solution in $L$. Then we can 
      find $t \in L$ such that \[ a^2=\frac{1-t^2}{1+t^2}, \hspace{3mm} b^2=\frac{2t}{1+t^2}.\] We will 
      analyze 
      the equation in two cases: 
            \begin{enumerate}
                 \item $t$ is in $K$
                  \item $t$ is not in $K$.  
             \end{enumerate}
       Assume $t$ is in $K$. If $a$ (resp. $b$) is in $K$, then $(a,b^2)$ (resp. $
       (a^2,b)$) gives a solution to the equation $x^4+y^2=1$ (resp. $x^2+y^4=1$) and Lemma \ref{diophantine} tells us that $(a,b)$ is trivial.
       If neither $a$ nor $b$ is in $K$, then $a=a_1 \sqrt{w}$ and $b=b_1\sqrt{w}$ for some $a_1, b_1$, 
       and $w$ in $K$ since $a^2,b^2 \in K$, hence $(a_1, b_1, 1/w)$ is a solution to $x^4+y^4=z^2$. 
       Again by Lemma \ref{diophantine}, this is not possible.
 
       Assume $t\notin K$. Since $t$ is in $L$, $K(t)$ is contained in $L$ and thus $L=K(t)$. Let $F(z)$ 
       be the minimal polynomial of $t$ over $K$ and let us define $X,Y$ as follows:
                        \[ 
                        X=(1+t^2)ab, \hspace{2mm} Y=(1+t^2)b 
                        \]
       Then, it is easy to see that $X^2=2t(1-t^2)$ and $Y^2=2t(1+t^2)$. Since $X,Y$ are in $L$, then 
       there are $c,d,e,f \in K$ such that 
                        \[ 
                        X=c+dt, \hspace{2mm} Y=e+ft
                         \]
        Let us define the polynomials $g(z)$ and $ h(z)$ as follows:
                          \[
                          g(z)=(c+dz)^2-2z(1-z^2)
                           \] 
                           \[ 
                           h(z)=(e+fz)^2-2z(1+z^2). 
                           \]
       Then $g(t)=h(t)=0$ because $X^2=2t(1-t^2)$ and $Y^2=2t(1+t^2)$.
       Since $F(z)$ is the minimal polynomial of $t$ over $K$, $F(z)$ must divide both $g(z)$ and 
       $h(z)$. 
       It follows that $g$ and $h$ both have exactly one root over $K$ (not necessarily the same root) 
       since 
       $g$ and $h$ are cubic polynomials. Let $u$ and $v$ denote the roots of $g$ and $h$ respectively, that is,          
                \[ 
                g(z)=2(z-u)F(z) \hspace{1mm} \text{and} \hspace{2mm} h(z)=-2(z-v)F(z). 
                \]
        Then $(-2u,2(c+d u))$ is a point on the affine curve $E_1: y^2=x^3-4x$. We previously computed the 
        points on $E_1(K)$ in the proof of Proposition \ref{diophantine}. There are 
        three of 
        them; $(0,0),(2,0),(-2,0)$, hence $u$ is either $0,1$ or $-1$.
        
       Similarly, $(2v,2(e+fv)$ is a point on the affine curve $E_2: y^2=x^3+4x$. Thus using our previous computations of $E_2(K)$, we see that the only possible solutions for $(2v,2(e+fv))$ are  
                \[ 
                \{(0,0),(2i,0),(-2i,0), (2,4),(2,-4),(-2,4i),(-2,-4i)\}.
                 \] 
      This shows that $v$ is either $0,\pm1$ or $\pm i$. We will first compute $F(z)$ for each value of $u$ using the 
        equality $g(z)=2(z-u)F(z)=(c+dz)^2-2z(1-z^2)$.
               \begin{enumerate}
            \item If $u=0$, then $g(0)=c^2=0$ and hence $c=0$. We obtain that $g(z)=2z(z^2+\frac{d^2}{2}
               z-1)$ and                    
                        \[ 
                        F(z)= z^2+\frac{d^2}{2}z-1.
                         \]                     
             We compute $h(z)=-2(z-v)(z^2+\frac{d^2}{2}z-1)$. Notice that the constant term of $h(z)
             $ 
             is $-2v$. On the other hand, $h(z)= (e+fz)^2-2z(1+z^2)$ and hence the constant term is 
             $e^2$. 
             Therefore $e^2=-2v$ and this is satisfied only if $v=0$ or $\pm i$.              
        \begin{enumerate}
               \item If $v=0$, then
                               \begin{equation} \label{v:0}
                               F(z)=z^2-\frac{f}{2}z+1
                               \end{equation} 
                                which contradicts with $F(z)= z^2+\frac{d^2}{2}z-1$ since the constant terms are 
                                not equal. 
                \item If $v=i$, then using the equality $h(i)=0$,
                                we find that   
                       \[ 
                       F(z)=z^2+z(-\frac{f^2}{2}+i)+i\frac{f^2}{2}=z^2+\frac{d^2}{2}z-1 
                       \]                       
                     and hence $-\frac{f^2}{2}+i= \frac{d^2}{2}$ and $ i\frac{f^2}{2}=-1$. It follows that $f^2=2i$ 
                      and $d=0$ and we arrive at a contradiction since $F$ has to be irreducible over $K$.
                 \item If $v=-i$, then we find that                        
                        \[ 
                        F(z)=z^2+z(-\frac{f^2}{2}-i)-i\frac{f^2}{2}=z^2+\frac{d^2}{2}z-1
                         \]                        
                     which forces $d$ to be $0$ and we arrive at a contradiction.    
          \end{enumerate}     
      \item If $u=1$, then $g(1)=(c+d)^2=0$ and hence $c=-d$. We see that 
      $g(z)=d^2(z-1)^2+2z(z^2-1)$   
            and we find that              
                 \[  
                 F(z)=z^2+z(\frac{d^2}{2}+1)-\frac{d^2}{2}
                 \]              
            which tells us that the constant term of $h(z)$ is $e^2=-d^2v$. This equation has solutions in $K$ only if 
              $v=0$ or $v=\pm -1$. 
            \begin{enumerate}
                 \item If $v=0$, we computed $F(z)$ in (\ref{v:0}). We obtain                     
                       \[ 
                        F(z)=z^2-\frac{f^2}{2}z+1=z^2+z(\frac{d^2}{2}+1)-\frac{d^2}{2}. 
                        \]  
                      The equality of the constant terms gives us $-d^2=2$ and we obtain a contradiction.   
                \item If $v=1$, then $(e+f)^2=2(1+1)=4$. Now the long division of $h(z)$ by $
                    (z-1)$gives us that    
                        \begin{equation}  \label{v:1}
                        F(z)=z^2-(\frac{f^2}{2}-1)z+\frac{e^2}{2}. 
                        \end{equation}
                   On the other hand, $F(z)=z^2+z(\frac{d^2}{2}+1)-\frac{d^2}{2}$, thus $e^2=-d^2=f^2$ and since $(e+f)^2=4$, we get $e^2=f^2=1$ and 
                      \begin{equation} \label{1/2,1/2}
                       F(z)=z^2+\frac{z}{2}+\frac{1}{2}. 
                       \end{equation}  
                \item Similarly if $v=-1$, we obtain $(e-f)^2=-4$. Similar to the previous part, the long division of $h(z)$ by $(z+1)$ produces                     
                          \begin{equation}\label{v:-1}
                           F(z)=z^2-(\frac{f^2}{2}+1)z-\frac{e^2}{2}.
                           \end{equation}
                      On the other hand, $F(z)=-2(z+1)(z^2+z(\frac{d^2}{2}+1)-\frac{d^2}{2})$.     
                   \[ 
                       e^2=d^2 \hspace{2mm} \text{and} \hspace{2mm} d^2+f^2=-4    
                        \]  
                        which implies that $e=0$ or $f=0$.  If $e=0$, then $d=0$ and we get a contradiction. 
                         If $f=0$, then $e^2=d^2=-4$ and we compute
                                             \begin{equation} \label{-1,1}
                                              F(z)=z^2-z+ 2. 
                                              \end{equation}
                \end{enumerate}      
                 
       \item If $u=-1$, then $g(-1)=(c-d)^2=0$ and hence $c=d$. We see that
            \[  
            F(z)=z^2+z(\frac{d^2}{2}-1)+\frac{d^2}{2}.
             \] 
               Constant coefficient of $h(z)$ equals to $2v (\frac{d^2}{2})=e^2$, thus $v$ has to be a 
               square 
                 in $K$. Therefore, $v$ can be $0$ or $\pm 1$ and we computed $F(z)$ for each of these values in (\ref{v:0}),(\ref{v:1}), and (\ref{v:-1}).
                               \begin{enumerate}
                   \item If $v=0$, then                   
                        \[ 
                         F(z)=z^2+z(\frac{d^2}{2}-1)+\frac{d^2}{2}=F(z)=z^2-\frac{f^2}{2}z+1. 
                         \] 
                        we find $d^2=2$ which is not possible since $d$ is in $K$.                 
                    \item If $v=1$, then                           
                         \[  
                         F(z)=z^2+z(\frac{d^2}{2}-1)+\frac{d^2}{2}=z^2-(\frac{f^2}{2}-1)z+\frac{e^2}{2} 
                         \]
                          which implies that $d^2=e^2$ and $d^2+f^2=4$. We see that $f=0$ and $d^2=4$. In this 
                          case, we compute 
                                 \begin{equation} \label{1,1}
                                  F(z)=z^2+z+2. 
                                  \end{equation}
                   
                   \item If $v=-1$, then we have 
                          \[
                           F(z)=z^2+z(\frac{d^2}{2}-1)+\frac{d^2}{2}=z^2-(\frac{f^2}{2}+1)z-\frac{e^2}{2}. 
                           \]
                        In this case, we compute that $d^2=1$ and hence
                          \begin{equation} \label{-1/2,1/2}
                           F(z)=z^2-\frac{z}{2}+\frac{1}{2}. 
                           \end{equation}        
              \end{enumerate}
   \end{enumerate} 
     To summarize, we showed that $F(z)$ is one of the following polynomials:
                    \[ 
                    z^2+\frac{z}{2}+\frac{1}{2}, \hspace{2mm}  z^2-z+ 2, \hspace{2mm}
                                         z^2+z+2, \hspace{2mm}          z^2-\frac{z}{2}+\frac{1}{2}.
                      \]                        
     One can easily check that the splitting field $L/K$ of each polynomial listed above is 
    $K(\sqrt{-7})$. Now we will find all non-trivial solutions to the equation $x^4+y^4=1$ over the field 
    $L=K(\sqrt{-7})$.   
    
    Remember that we started with a solution $(a,b)$ in $L$ and constructed $X$ and $Y$. It is easy to 
    see that $a=X/Y=\frac{c+dt}{e+ft}$ and $b=Y/(1+t^2)=\frac{e+ft}{1+t^2}$. In the following, we use the notation $w=\sqrt{-7}$ for simplicity. Also $\delta_j$ 
    for $j=1,2,3$ denote the integers such that $\delta_j^2=1$.
        \begin{enumerate} 
              \item We found $F(z)= z^2+z+2$ in (\ref{1,1}) with conditions that $f=0$, $d^2=e^2=4$ and $c=d$. The roots 
                   of   
                   $F(z)$ are $t=(-1+\delta_3 w)/2$ and
               we compute 
                           \[
                            a= \delta_1 (1+\delta_3 w)/2.                                     \] 
               Then we find $b$ as $\delta_2 (-1+\delta_3 w)/2$.   
                              
           \item  Similarly, we obtained $F(z)=z^2-z+2$ in (\ref{-1,1}) with the conditions $c=-d$, $e^2=d^2=-4$ and $f=0$.
                        In this case, we find $t=(1+\delta_3 w)/2$ and 
                     \[ 
                     a=\delta_1(-1+\delta_3 w)/2.    
                      \]   
                  Similar to the first case, $b= i\delta_2(1+\delta_3 w)/2$.      
           \item The polynomial $F(z)=z^2+\frac{z}{2}+\frac{1}{2}$ we found in (\ref{1/2,1/2}) produces $t=(-1 \pm w)/4$ and we 
                     compute
                    \[
                     a=\delta_1 i (-1+\delta_3 w)/2,  \hspace{3mm} b= \delta_2 (1+\delta_3 w)/2.
                     \]   
                       
          \item Similarly, the polynomial $F(z)=z^2-\frac{z}{2}+\frac{1}{2}$ in (\ref{-1/2,1/2}) produces $t=(1 \pm w)/4$ 
                  and               
                   \[ 
                   a=\delta_1 i (1+\delta_3 w)/2, \hspace{3mm} b=\delta_2 i(-1+\delta_3 w)/2.
                   \] 
     \end{enumerate}
  Hence the result follows.
  
   \end{proof}
    \begin{remark}
        Unfortunately, the method of the proof of Theorem \ref{64} does not apply to the solutions over 
        the field $\Q(\sqrt{-3})$ since there 
        are non-trivial solutions to the equation $u^4+v^4=z^2$ over $\Q(\sqrt{-3})$.    
      \end{remark}


\section{Restrictions on the Torsion Subgroups} \label{restrictions}
In the proof of Proposition \ref{20}, we will explicitly describe three cyclic $20$ isogenies. To show that their field of definition is $\Q(i)$, we will first need the following result. 

\begin{prop} \label{i}
                   Let $E_1$ and $E_2$ be the elliptic curves associated to the lattices $[1,i]$ and $[1,2i]$.        
                   Then any isogeny in $\text{Hom}(E_1,E_2)$ is defined over the field $\Q(i)$.
\end{prop}
\begin{proof} 

Let $\lambda :E_1 \to E_2$ be an isogeny, then $\lambda=[\alpha]_{E_1,E_2}$ for some $\alpha \in \Co$ with $\alpha$ and $\alpha i$ both in $[1,2i]$ since $\alpha [1,i]$ is contained in $[1,2i]$. It follows that $\alpha=2a+2bi$ for some integers $a$ and $b$. Hence $\Hom(E_1,E_2)$ is isomorphic to $2(\Z[i])$ as an additive group and it is enough to show that the isogenies $[2]_{E_1,E_2}$ and $[2i]_{E_1,E_2}$ are defined over $\Q(i)$. 

We first note that both the lattices $[1,i]$ and $[1,2i]$ have $j$-invariant in $\Q$, since their endomorphism rings $\Z[i]$ and $\Z[2i]$ have class number one. Hence we may assume that $E_1,E_2$ are defined over $\Q$.

  
We see that $[2]_{E_1,E_2}$ and $[2i]_{E_1,E_2}$ both define isogenies of degree $2$. Since $E_1$ has $j$ invariant $1728$, it has a model of the form 
       \[ 
       y^2=x^3+dx 
       \] 
 for some $d$. The elliptic curve $E_1$ has at least one $2$-isogeny defined over $\Q$ since $(0,0)$ is a point on $E_1$. Note that  $[6+2i]_{E_1,E_2}$ is a cyclic isogeny of degree $20$ in $\Hom(E_1,E_2)$. If $[2]_{E_1,E_2}$ and $[2i]_{E_1,E_2}$ are both defined over $\Q$, then $[6+2i]$ is also defined over $\Q$ but there is no elliptic curve over $\Q$ with a rational cyclic $20$-isogeny. See \cite[Theorem~2.1]{Laska-Lorenz}. Hence the isogenies $[2]_{E_1,E_2}$ and $[2i]_{E_1,E_2}$ cannot be both defined over $\Q$.

 Assume $[2i]_{E_1,E_2}$ is not defined over $\Q$, then it is defined over $L$ (more precisely $\Q(\sqrt{-d})$). We will show that the field $L$ is $\Q(i)$. 
 Composing the following isogenies of $E_1$ and $E_2$
  \[ 
 E_1 \overset{[2i]}{\longrightarrow} E_2  \overset{[i]}{\longrightarrow} E_1,
 \]
 we obtain the endomorphism $[-2]_{E_1}$ of $E_1$ which is defined over $\Q$. This implies that $[i]_{E_2,E_1}$ is defined over $L$. Then the endomorphism
        \[
          E_1 \overset{[2]}{\longrightarrow} E_2 \overset{[i]}{\longrightarrow} E_1 \hspace{2mm} 
          \]
  is defined over $L$. On the other hand, the endomorphism given above is defined over $\Q(i)$ since the field of definition of $E_1 \overset{[i]}{\longrightarrow} E_1$ is $\Q(i)$ and $E_1\overset{[2]}{\longrightarrow} E_1$ is defined over $\Q$. Thus $L$ must equal to $\Q(i)$. (Notice here that the endomorphism $[2]_{E_1}$ is the multiplication by $2$ map on $E_1$ and hence its field of definition is $\Q$.)
  
One can do a similar discussion for the case $[2i]_{E_1,E_2}$ is defined over $\Q$.

\end{proof}

 \begin{prop}\label{20} 
       Let $E$ be an elliptic curve over $K=\Q(i)$ and let $F$ be the maximal elementary abelian 
       $2$-extension of $K=\Q(i)$. Then $E(F)$ cannot have a cyclic $K$-rational subgroup of order $20$. 
   \end{prop}
      
      \begin{proof}
                       
            Let $E_1$ and $E_2$ be the elliptic curves with the endomorphism rings                           
                             \[ 
                             \Z[i] \hspace{2mm} \text{and} \hspace{2mm} \Z[2i] 
                             \]
            respectively and let $\lambda_i$ be the following endomorphisms                           
                            \[ 
                            \lambda_1=[4+2i] : E_2 \to E_2, 
                            \]
                            \[ 
                            \lambda_2= [6+2i]: E_1\to E_2, 
                            \]
                            \[ 
                            \lambda_3= [6-2i]: E_1 \to E_2. 
                            \]   
for $i=1,2,3$. Let $\bar{\lambda_i}$ denote the dual of the isogeny $\lambda_i$ for each $i=1,2,3$. 

The elliptic curves $E_1$ and $E_2$ have $j$-invariant in $\Q$ since their endomorphism rings are orders of class number one in the field $\Q(i)$ and hence the isogeny $\lambda_1$ is 
              defined over $\Q(i)$ by \cite[Theorem~2.2]{silverman-advance} and $\lambda_2, \lambda_3$ are defined over $\Q(i)$ by Proposition \ref{i}.              
            Consequently, the $6$ non-cuspidal points (see \S\ref{modular curve 20}) on $X_0(20)$ over $K$ represent the elliptic curves $E_1$ and $E_2$ up to isomorphism, that is, if an elliptic curve defined over $K$ has a cyclic $K$-rational $20$-isogeny, it is isomorphic to $E_1$ or $E_2$.
    
    Let $E/K$ be an elliptic curve such that $E(F)$ contains a cyclic $K$-rational subgroup $C$ of order $20$. Then $E$ is isomorphic to $E_1$ or $E_2$.
                        
  Assume $E$ is isomorphic to $E_1$, then $E$ has $j$-invariant $1728$ and a quadratic twist $E^{(d)}(K)$ of $E$ has a point of order $5$ by Lemma \ref{Laska-Lorenz}, a contradiction to Theorem \ref{Burt-0,1728} since $E^{(d)}$ also has $j$-invariant $1728$.                   
 
  Assume now that $E$ is isomorphic to $E_2$. The $j$-invariant of $E_2$ is given in \cite{Cox} as 
            \[
            j(E_2)=11^3.
            \]
   We find a model for the elliptic curve with this $j$-invariant in \cite{LMFDB} as:
       \[   
           E': y^2=x^3-11x-14.
       \]
  The elliptic curve with the given equation has good reduction at $3$ and it has $64$ points over the finite field $F_{81}$. Since $E'(F)[5]$ has to inject into $E'(\F_{81})$ and $5$ does not divide $64$, $E'(F)$ can not have a point of order $5$.
  
 Now, the elliptic curve $E$ is a quadratic twist of $E'$ and $E'(F) \simeq E(F)$, hence we arrive at a contradiction.      
 \end{proof}
         
We use the following result to show that there exists no elliptic curve $E$ defined over $\Q(i)$ such that $E(F)$ contains a $K$-rational subgroup isomorphic to $\Z/2\Z \oplus \Z/32\Z$ (Proposition \ref{can not happen}). 

 \begin{prop}\label{j}
       Let $E$ be an elliptic curve defined over $\Q(i)$. Assume $E$ has a cyclic isogeny of degree 
       $64$ defined over a quadratic extension of $\Q(i)$, then the j-invariant of $E$ is integral.
 \end{prop}
 
 \begin{proof}
     Kenku \cite[Proposition~2]{Kenku64} shows the modular curve $X_0(64)$ has the affine equation
                         \[ 
                         x^4+y^4=1 
                         \]
      and that it has $12$ cusps; the points at infinity \cite[p.1]{Rohrlich}.
      
      Theorem \ref{64} proves that any point of $Y_0(64)$ defined over a quadratic 
      extension of $\Q(i)$ is in $\Q(\sqrt{-7},i)$ and there are $32$ such points.    
        
        Let $E_i$ for $i=1,2$  be the elliptic curves with complex multiplication by the orders 
        $ \Z[\frac{1+\sqrt{-7}}{2}]$ and $\Z[\sqrt{-7}]$ respectively and let $E_3$ be the elliptic 
        curve associated to the 
        lattice $[1,2\sqrt{-7}]$. Notice that $E_3$ also has complex multiplication by the order 
        $\Z[2\sqrt{-7}]$. Now let $\alpha_i$ for $i=1,2,3,4$ as follows:                       
                        \[
                         \alpha_1=\frac{9+5\sqrt{-7}}{2}, \hspace{3mm}
                               \alpha_2= 1+3\sqrt{-7},  
                          \]
                        \[    
                         \alpha_3= 6+2\sqrt{-7}, \hspace{3mm}
                               \alpha_4= 10+2\sqrt{-7}
                         \]
    Then $\alpha_i$ defines an endomorphism $\lambda_i=[\alpha_i]_{E_i}$ of $E_i$ for every $i=1,2,3$ and $\alpha_4$ defines 
    an isogeny $\lambda_4=[\alpha_4]_{E_2,E_3}$ from $E_2$ to $E_3$. Let $\bar{\lambda_i}$ be the dual of $\lambda_i$ for each $i=1,2,3,4$ and let $C_i$ be the kernel of $\lambda_i$. Similarly let $\bar{C_i}$ be the 
    kernel of $\bar{\lambda_i}$. The $j$-invariants of $E_1$ and $E_2$ are given in \cite{Cox} as  
                  \[
                   j(E_1)=-3^35^3 \hspace{2mm} \text{and} \hspace{2mm} j(E_2)= 3^35^317^3. 
                    \]
   The isogeny $\lambda_i$ is cyclic of degree $64$ for each $i$ and \cite[Theorem~2.2]{silverman-advance} tells 
     us that $\lambda_1$ and $\lambda_2$ are defined over $\Q(\sqrt{-7})$. See also 
     \cite[Lemma~1]{Kenku64}.              
     
     The class number of the order $\mathcal{O}=\Z[2\sqrt{-7}]$ can be computed by the formula in 
     \cite[Theorem~7.24]{Cox}; the conductor $f$ of $\mathcal{O}$ is $4$ and the class number of 
     $\mathcal{O}$ is $h=2$, hence $E_3$ can be defined over a quadratic extension of 
     $\Q$. Let $L=\Q(\sqrt{-7})(j(E_3))$ be the ring class field of the order $\mathcal{O}$. We know the 
     following facts: (See also \cite[Proposition~9.5, p.184]{Cox}.)
       \begin{enumerate}
           \item The minimal polynomial $f$ of $j(E_3)$ is in $\Z[x]$ and it is of degree 
                    $h(\mathcal{O})=2$.
           \item All primes of $\Q(\sqrt{-7})$ that ramify in $L$ must divide $4\Z[\frac{1+\sqrt{-7}}{2}]$.
        \end{enumerate}   
     First property tells us that $L=\Q(\sqrt{-7}, \sqrt{d})$ for some square-free $d\in \Q$. Consider the following diagram of field extensions. 
    
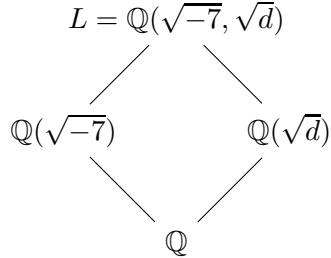
\[
     \begin{tikzpicture}[node distance = 1.5cm, auto][scale=0.05]
      \node (Q) {$\mathbb{Q}$};
      \node (E) [above of=Q, left of=Q] {$\Q(\sqrt{-7})$};
      \node (F) [above of=Q, right of=Q] {$\Q(\sqrt{d})$};
      \node (L) [above of=Q, node distance = 3cm] {$L=\Q(\sqrt{-7},\sqrt{d})$};
      \draw[-] (Q) to node {} (E);
      \draw[-] (Q) to node {} (F);
      \draw[-] (E) to node {} (L);
      \draw[-] (F) to node {} (L);
      \end{tikzpicture}
      \]    
    
     If a prime in $\Q$ ramifies in $\Q(\sqrt{d})$, then it is either $2$ (by the second property given above) or $7$ (by the fact that the only prime ramifies in $\Q(\sqrt{-7})$ is $7$). Since $L$ contains $\sqrt{-7}$, we may assume that $d$ and $7$ are relatively prime. Hence $d$ is $\pm2$ or $-1$ and      
     
            \[ 
            L=\Q(\sqrt{-7},\sqrt{\pm2}) \hspace{2mm} \text{or} \hspace{2mm} L=\Q(\sqrt{-7},i). 
            \]         
    We will use \cite[Theorem~9.2, p.180]{Cox} to determine $L$. Let $p=29$. Then $p$ can be 
    written as $p=1^2+28.1^2$. We have the following:
    \begin{enumerate}
     \item  $-28$ is a quadratic residue modulo $29$.
    \item  The polynomial $x^2+1$ has a solution modulo $p$ where as $x^2 \pm 2$ 
    does not. We see that 
        \[  
           L=\Q(\sqrt{-7},i).
        \] 
        
     \end{enumerate}   
    This shows that we may assume $E_3$ is defined over the field $\Q(i)$. Thus by Theorem \cite[Theorem~2.2]{silverman-advance}, 
     $\lambda_3$ is defined over $\Q(\sqrt{-7},i)$. Now we will show that $\lambda_4$ 
     is also 
     defined over $\Q(\sqrt{-7},i)$.
     
     Similar to Proposition \ref{i}, we can show that $\text{Hom}(E_2,E_3)$ is isomorphic to 
     $2(\Z[\sqrt{-7}])$ and it is generated by the isogenies $[2]_{E_2,E_3}$ and $[2\sqrt{-7}]_{E_2,E_3}$. It follows that these 
     isogenies cannot be both defined over $\Q(i)$ since they generate $\text{Hom}(E_2,E_3)$. 
     Otherwise $[10+2\sqrt{-7}]_{E_2,E_3}$ would be a cyclic $64$-isogeny over $\Q(i)$ but there are no such isogenies 
     defined over $\Q(i)$ by Proposition~\ref{modular curves}. Since $[2\sqrt{-7}]_{E_2}$ is an endomorphism 
     of $E_2$ defined over $\Q(\sqrt{-7})$ (and not over $\Q$), similar to Proposition \ref{i}, we see that either $[2]_{E_2,E_3}$ or 
     $[2\sqrt{-7}]_{E_2,E_3}$ is defined over $\Q(i,\sqrt{-7})$.
     
     Thus $(E_i, C_i)$ and $(E_i,\bar{C_i})$ are rational over $L$ and they represent $8$ 
     different points on $Y_0(64)(L)$.
     
     Let $(E_i,C_i)$ correspond to the point $(X(w_i),Y(w_i))$ on                                      
                                         \[
                                          x^4+y^4=1. 
                                          \]
     Denote by $W_{64}$ the modular transformation corresponding to the matrix     
                        \[  
                         \begin{bmatrix} 0 & -1 \\ 64 & 0  \end{bmatrix}.
                          \]  
       Then $W_{64}$ acts as an involution on $X_0(64)$. Furthermore, we know by 
       \cite[Proof of Lemma 1]{Kenku64} that $W_{64}$ sends $(E_i,C_i)$ to $(E_i,\bar{C_i})$. Using the 
       transformation formula for $\eta(z)$, Kenku shows that $W_{64}$ maps $(x,y)$ to $(y,x)$. Hence
       if $(E_i,C_i)$ correspond to the point $(X(w_i),Y(w_i))$ on 
                                     \[ 
                                     x^4+y^4=1, 
                                     \]
       then $(E_i, \bar{C_i})$ correspond to $(Y(w_i),X(w_i))$. Using the transformation 
         formula for $\eta(z)$ again, we see that
              \[ 
              \begin{bmatrix} 1 &0 \\ 16 & 1 \end{bmatrix} : \hspace{2mm} (x,y) \mapsto (-ix,y) 
              \] 
              \[  
              \begin{bmatrix} 1 &0\\ 32 & 1\end{bmatrix} : \hspace{2mm} (x,y) \mapsto (-x,y)
               \]
              \[ 
               \begin{bmatrix} 1 &0\\ 48 & 1\end{bmatrix} : \hspace{2mm} (x,y) \mapsto (ix,y). 
               \]
     Since these matrices are in $\SL_2(\Z)$, their actions do not change the 
     isomorphism class of the elliptic curve corresponding to the given point on $Y_0(64)$, hence the remaining $24$ 
     points on $Y_0(64)$ (over $(K(\sqrt{-7}))$) correspond to some pair $(E,C)$ where $E$ is isomorphic to 
     $E_1,E_2$ or $E_3$. 
    
    This proves that an elliptic curve with a cyclic isogeny of degree $64$ over a quadratic extension of 
    $\Q(i)$ has complex multiplication and thus its j-invariant is integral.    
     
 \end{proof}
    
    \begin{remark}
       Over $K=\Q(\sqrt{-3})$, we find a point $(2/\sqrt{5},\sqrt{-3}/\sqrt{5})$ on the curve $x^4+y^4=1$ 
       and a computation on Magma shows that it produces an elliptic curve with a non-integral
        $j$-invariant.
    \end{remark}
        
    \begin{prop} \label{can not happen} 
         Let $K=\Q(\sqrt{d})$ for $d=-1$ or $-3$ and let $F$ be the maximal elementary abelian extension 
         of $K$. Assume that $E$ is an  elliptic curve defined over $K$. Then $E(F)$ does not contain a 
         rational subgroup isomorphic to one of the following groups:            
                    \begin{align*} 
                       \Z/4\Z\oplus \Z/2\Z \oplus \Z/3\Z \oplus  \Z/3\Z, & \hspace{3mm}
                       \Z/2\Z \oplus \Z/8\Z \oplus \Z/3\Z, \\  \Z/2\Z \oplus \Z/4\Z \oplus \Z/5\Z.&     
                    \end{align*}
         Moreover, if $K=\Q(i)$, then $E(F)$ does not contain a $K$-rational subgroup isomorphic to the 
          group
                                       \[ 
                                       \Z/32\Z\oplus \Z/2\Z.
                                       \]
    \end{prop}

      \begin{proof}
          For each group except $\Z/32\Z\oplus \Z/2\Z$ given in the statement of the proposition, the proof 
          follows the proof of \cite[Proposition 2.4]{Laska-Lorenz} by using Proposition \ref{modular 
          curves}, Proposition \ref{20}, Theorem \ref{modular curves-Burton}, and Remark \ref{explain Burton}.
          
          Assume that $E(F)$ has a $K$-rational subgroup $V\simeq \Z/32\Z \oplus \Z/2\Z
          $. We will slightly modify the argument given in the proof of Proposition 2.4 in 
          \cite{Laska-Lorenz}. Let $A$ denote the subgroup of $\Aut(V)$ given by the action of 
          $\Gal(\bar{K}/K)$ on 
          $V$. Writing $V$ as $L \oplus S$, with $L\simeq \Z/32\Z$ and $S\simeq \Z/2\Z$, we know by 
          Lemma 1.5 in \cite{Laska-Lorenz} that there exist a subgroup $A_0$ of $A$ with index $2$ which stabilizes $L$ and $S$. Now                          
                                \[   
                                 \phi: E\rightarrow E'=E/S  
                                \]
         gives us a pair $(E',L')$ where $L'={\phi'}^{-1}(L)$ is of order $64$ and $\phi'$ is the 
          dual isogeny of $\phi$. Thus, the pair $(E',L')$ corresponds to a point on 
          $Y_0(64)$, defined over a quadratic extension of $K$. By Lemma \ref{j}, the $j$-invariant of $E'$ is integral and thus $E'$ has either 
          good or additive reduction \cite[Proposition 5.5, p.181]{Silverman}.  By 
          \cite[Corollary 7.2, p.185]{Silverman}, $E$ also has either good or additive reduction.
                   
          Let $L$ be the field generated by $V \subset E(F)$, let $\mathfrak{P}$ be a prime of 
          $\mathcal{O}_L$ above $\mathfrak{p}=2-i$ and let $\tilde{E}$ denote the reduction of $E$ at 
          prime $\mathfrak{P}$. 
          If $E$ has good reduction at $\mathfrak{P}$, then $E(L)_{(2)}$ injects into $\tilde{E}(\F_{25})$. By the 
          Weil conjectures, we see that $\left |\tilde{E}(\F_{25})\right |\leq 36$ which is not possible since 
          $E(L)$ must have at least $64$ points. 
          Assume $E$ has additive reduction at $\mathfrak{P}$. Let $\tilde{E}$ denote the reduction of 
            $E$ and let $\tilde{E}_{\text{ns}}$ denote the set of non-singular points of $\tilde{E}$.  
            Define                              
                        \[ 
                        E_1(L)=\{ P \in E(L) : \tilde{P}=\tilde{O}     \}       
                        \]
             and     
                         \[ 
                         E_0(L)=\{ P \in E(L) :  \tilde{P} \in \tilde{E}_{\text{ns}} \}
                         \]
By 
            \cite[{Proposition 3.1(a), p.176}]{Silverman}, $E_1(L)$ has no non-trivial points of order $2^k$  
            for any $k$. Then $E_0(L)_{(2)}$ injects into $\tilde{E}_{ns}(\F_{25})$. However, 
           $\tilde{E}_{ns}(\F_{25})$ is isomorphic to the additive group $\F_{25}$ and hence 
           $E_0(L)_{(2)}$ is trivial.  We know by \cite[{Addendum to Theorem 3, \textsection{6}}]{Tate} that                                  
                                       \[ 
                                        [E(L) :E_0(L)]  \leq 4.
                                       \]            
             Therefore, $E(L)_{(2)}$ has at most $4$ elements which contradicts our 
             assumption that $E(L)$ has a subgroup isomorphic to $\Z/2\Z  \oplus  \Z/32\Z$.

       \end{proof}

 
  \begin{theorem} \label{Laska}
  Let $K$ be a quadratic cyclotomic field and let $F$ be the maximal elementary abelian extension of $K$.   
  Assume that $E$ is an elliptic curve defined over $K$.
            \begin{enumerate} 
                \item If $K=\Q(i)$,
                   then $E(F)_{\text{tors}}$ is isomorphic to one of the following groups:
                      \begin{align*}
                          & \Z/2^{b+r}\Z \oplus \Z/2^b\Z & &    (b=1,2,3  \hspace{1mm} \text{and}\hspace{1mm} 
                          r=0,1,2,3) & &\\
                          & \Z/2^{b+r}\Z \oplus \Z/2^b\Z \oplus \Z/3\Z & &  (b=1,2,3 \hspace{1mm} \text{and}
                          \hspace{1mm} r=0,1) & &\\
                          & \Z/2^{b}\Z \oplus \Z/2^b\Z \oplus \Z/3\Z \oplus \Z/3\Z & &  ( b=1,2,3) & &\\
                          & \Z/2^{b}\Z \oplus \Z/2^b\Z \oplus \Z/5\Z & &  (b=1,2,3) & &
                       \end{align*}
                      or ${1}, \Z/3\Z$, $ \Z/5\Z$, $ \Z/7\Z$, $ \Z/9\Z$, $ \Z/15\Z$, and $\Z/3\Z \oplus \Z/3\Z$ .
               
               \item 
                   If $K=\Q(\sqrt{-3})$, then $E(F)$ is either isomorphic to one of the above groups or 
                       \[ \Z/2\Z \oplus \Z/32\Z, \hspace{2mm} \Z/4\Z \oplus \Z/64\Z  \]

            \end{enumerate}
    \end{theorem}
       \begin{proof} 
            The proof follows from \cite[Proof of Theorem 2.5]{Laska-Lorenz} using Proposition \ref{can not 
            happen}, Proposition \ref{modular curves}, Lemma \ref{20}, Proposition \ref{odd} and Theorem 2 in \cite{Najman-cyclotomic}.
       \end{proof}
\section{$E(K)_{\text{tors}}$ is Non-Cyclic}\label{full torsion}
In this section, we study $E(F)_{\text{tors}}$ when the torsion subgroup of $E$ over $K$ is not cyclic.  Theorem 2 in \cite{Najman-cyclotomic} shows that if $E(K)_{\text{tors}}$ is not cyclic, then it is isomorphic to one of the following groups.
  \[    
      \Z/2\Z  \oplus \Z/2n\Z \hspace{2mm} \text{for}\hspace{2mm}  n=1,2,3,4 \hspace{2mm} \text{and} \hspace{2mm} (\text{only if} \hspace{2mm} K=\Q(i))\hspace{2mm}  \Z/4\Z \oplus \Z/4\Z .
  \] 
We study each of these groups separately. The Proposition \ref{2,8}, Proposition \ref{2,6}, Proposition \ref{4,4}, Proposition \ref{2,4} and Proposition \ref{2,2} are the main results of this section.     
 
 Let $E$ be an elliptic curve over a number field $k$ with $E[2]\subset E(k)$. We know that the    
    points of order $2$ on $E$ are given by the roots of the polynomial $f(x)$ where $E$ is defined by 
    $y^2=f(x)$. 
    Therefore, we may assume that $E$ is of the form 
              \[  y^2=(x-\alpha)(x-\beta)(x-\gamma) \]
       with $\alpha, \gamma, \beta$ in $k$. We will use the following results very often in this section. 
     Remember that $F$ denotes the maximal elementary abelian extension of $K$.
         
          \begin{lemma}[{\cite[Theorem 4.2]{Knapp}}]
               \label{Knapp}
               Let $k$ be a field of characteristic not equal to $2$ or $3$ and let $E$ an elliptic curve over $k$   
                   given by 
                         \[ 
                         y^2=(x-\alpha)(x-\beta)(x-\gamma)  
                          \]
                     with $\alpha,\beta,\gamma$ in $k$. 
                    For 
                   $P=(x,y)$ in $E(k)$, there exists a $k$-rational point $Q$ on $E$ such that $[2]Q=P$ if 
                   and only if $x-\alpha, x-\beta$ and $x-\gamma$ are all squares in $k$. 
                
                 In this case if we fix the sign of the square roots of $x-\alpha, x-\beta,x-\gamma$, 
            then the $x$-coordinate of $Q$ equals to either
                \[    
                \sqrt{x-\alpha} \sqrt{x-\beta} \pm \sqrt{x-\alpha}\sqrt{x-\gamma} \pm \sqrt{x-\beta}\sqrt{x- 
                      \gamma}+x
               \] 
         or
               \[    
              - \sqrt{x-\alpha} \sqrt{x-\beta} \pm \sqrt{x-\alpha}\sqrt{x-\gamma} \mp \sqrt{x-\beta}\sqrt{x- 
                      \gamma}+x.
               \]
     See also the proof of Theorem 4.2 in \cite{Knapp}.
      \end{lemma}
    \begin{theorem} [{\cite[Theorem 9]{Burton-even}}] \label{Ono}
         Let $K$ be a number field and let $E/K$ be an elliptic curve with full $2$-torsion. Then $E$ has a 
         model of the form $y^2=x(x+\alpha)(x+\beta)$ where $\alpha, \beta \in \mathcal{O}_K$.
            \begin{enumerate}
               \item $E(K)$ has a point of order $4$ if and only if $\alpha,\beta$ are both squares, $-\alpha, \beta-\alpha$ are both squares, or $-\beta, \alpha-\beta$ are 
                     both squares in $K$.
               \item $E(K)$ has a point of order $8$ if and only if there  exist a $d \in \mathcal{O}_K$,  
                     $d \neq0$ and a Phytogorean triple $(u,v,w)$ such that 
                      \[    \alpha=d^2u^4, \hspace{1mm} \beta=d^2v^4, \]    or we  
                      can replace $\alpha,\beta$ by $-\alpha,\beta-\alpha$ or $-\beta,\alpha-\beta$ as in the 
                      first case.
          \end{enumerate}
    \end{theorem}
    
    By Theorem \ref{Ono}, we may assume that an elliptic curve $E$ with full $2$-torsion has the model 
    $y^2=x(x+a)(x+b)$ with $a,b \in \mathcal{O}_K$. We denote this curve by $E(a,b)$. Then $E(a,b)$ 
    is isomorphic (over $K$) to $E(-a,b-a)$ and $E(-b,a-b)$  by the isomorphisms 
              \[ 
              (x,y) \mapsto (x+a,y) \hspace{1mm} \text{and} \hspace{1mm} (x,y) \mapsto (x+b,y)
               \]
    respectively. 
    Assume $E$ has a point of order $4$ in $E(F)$. Then using Theorem \ref{Ono} together with the  
    isomorphisms between $E(a,b), E(-a,b-a)$, and $E(-b,a-b)$, we may assume that there is a point 
    $Q$ such that $[2]Q=(0,0)$ and that $a$ and $b$ are both squares. 
    Notice that with a similar discussion, we may assume that if $E(K)$ has a point of order $8$, then  
    $M=u^4$ and $N=v^4$ for some $u,v \in \mathcal{O}_K$ such that $u^2+v^2$ is a square in $K
    $ (replacing $E$ by a quadratic twist if necessary). 
    
    The following result is on the classification of twists of elliptic curves 
    over $K$. We will use this result very often in this section.
        \begin{theorem}[{\cite[Theorem 15]{Burton-even}}]
               \label{Burt}
                  Let $K=\Q(\sqrt{D})$ with $D=-1,-3$, $d\in K$ a non-square and let $E/K$ an elliptic curve   
                     with full $2$-torsion. Then,
             \begin{enumerate}
                \item If $E(K)_{tor}\simeq \two \oplus \eight$, then            $E^d(K)_{tor} \simeq \two \oplus \two
                $.
                 \item If $E(K)_{tor}\simeq \two \oplus \six$, then              $E^d(K)_{tor} \simeq \two \oplus \two
                 $.
                 \item If $E(K)_{tor}\simeq \four \oplus \four$, then             $K=\Q(i)$ and $E^d(K)_{tor} \simeq 
                             \two \oplus \two$.
                 \item If $E(K)_{tor}\simeq \two \oplus \four$, then              $E^d(K)_{tor} \simeq \two \oplus 
                             \two$ unless               $K=\Q(\sqrt{-3})$ and $d=-1$ in which case                 
                              $E^d(K)_{tor} \simeq \two \oplus \four$ or $ \two            \oplus \two$.
                 \item If $E(K)_{tor}\simeq \two \oplus \two$, then $E^d(K)_{tor} \simeq \two \oplus \two$ for 
                         almost all $d$.
           \end{enumerate}
      \end{theorem}
       For the rest of \S\ref{full torsion}, $E$ will always denote an elliptic curve
       given by the equation          
                 \[ 
                 E(a,b): y^2=x(x+a)(x+b), \hspace{2mm} a,b\in K. 
                 \] 
     Furthermore, we will assume that the greatest common divisor $(a,b)$ (defined up to a unit) is    
     square-free. Otherwise, we may replace it by the quadratic twist $E^{(d^2)}$ of $E$ where $d^2$ 
     divides both $a$ and $b$.    
     Also, $[n]$ denotes  the multiplication by $n$ on the elliptic curve $E$. 
     
     We will need the following lemma for the proof of Proposition \ref{2,8}.
     
      \begin{lemma}\label{square}
               Let $\alpha$ be in $K$. If $\sqrt{\alpha}$ is a square in $F$, then $\alpha$ or $-\alpha$ is a 
               square in $K$. If $K=\Q(i)$, then $\alpha$ must be a square in $K$.
           \end{lemma}
            
           \begin{proof}
               Let $w\in F$ be such that $w^2=\sqrt{\alpha}$. Then $w$ is a root of the polynomial 
               $f(x)=x^4-\alpha$ defined over $K$. Since $f$ has one root in $F$ and
               $F$ is a Galois extension of $K$,  $f$ splits in $F$. If $f$ is reducible over $K$, then it has to 
               be product of 
               two quadratic polynomials, otherwise $3$ has to divide the order of $Gal(F/K)$. Let us write 
               $f$ as a product of two polynomials.                 
                            \[
                            x^4-\alpha=(x^2+ex+f)(x^2+cx+d)
                             \] 
                 for $e,f,c,d \in K$.  
               Then $ec+f+d=0$, $ed+fc=0$ and $c+e=0$. Therefore, replacing $e$ by $-c$, we obtain 
               $c(f-d)=0$. So, either $c=0$ or $f=d$. If $c=0$, then  $f=-d$ and so $fd=-\alpha$ implies that 
               $\alpha=f^2$ in $K$. If $f=d$, then $\alpha=-fd$, so $-\alpha$ is a square in $K$.
               
               Now assume that $f$ is irreducible over $K$. The Galois group of $f$ over $K$ is an 
               elementary abelian $2$-group since it is a quotient of $\Gal(F/K)$. However, an
               elementary abelian 2-subgroup of $S_4$ is either of order $2$ or the Klein four-group 
               $V_2$. Hence 
               it must be isomorphic to $V_2$ and in particular, it has order $4$. 
               If $f$ remains irreducible over $K(i)$, then we obtain an automorphism of 
               order $4$ in $\Gal(K(w)/K(i))$; namely $ w\mapsto iw $ which contradicts the fact that Galois 
               group is $V_2$. Therefore, if $K=\Q(i)$, it can not be irreducible and we may assume that 
               $K=\Q(\sqrt{-3})$. 
               Then $f$ has to be reducible over $K(i)$ and by our previous discussion, we see that $
               \alpha$ or $-\alpha$ is a square in $K(i)$. If $\alpha=\pm d^2$ with $d\in K(i)$, then 
               $d=bi$ for some $b\in K$ and $\alpha=\pm b^2$.
          
          \end{proof}
    
    \begin{prop}\label{2,8} 
    Let $E: y^2=x(x+a)(x+b)$ be an elliptic curve defined over $K$. Assume that 
                               $E(K)_{\text{tors}}\simeq        \two \oplus \eight$. 
               \begin{enumerate}
                    \item If $K=\Q(i)$, then $E(F)_{\text{tors}}$ is either isomorphic to $\four \oplus \sixteen$  
                            or  $\four \oplus \thirtytwo$.
                    \item If $K=\Q(\sqrt{-3})$, then $E(F)$ is isomorphic to one of the groups above or $\four 
                           \oplus \Z/64\Z$.
              \end{enumerate}
     \end{prop}

    \begin{proof}
        Since any number in $K$ is a square in $F$, $\Z/4\Z \oplus \Z/16\Z \subset E(F)$ by Lemma  
         \ref{Knapp}. By Theorem \ref{Ono}, we may assume that $a=u^4$ and $b=v^4$ for some $u,v \in 
         \mathcal{O}_K$ such that $u^2+v^2=w^2$ for some $w \in K$. We will show that $\eight \oplus 
         \eight \not\subset E(F)$. Let $Q_2=(x,y)$ be a point of order $4$ such that $[2]Q_2=(-a,0)$. By 
         Lemma \ref{Knapp}, we compute that $x$ equals to one of the followings:       
                      \[ 
                      \pm \sqrt{-u^4+0}\sqrt{-u^4+v^4}-u^4.
                      \]    
        If $\eight \oplus \eight \subset E(F)$, then there is a point $Q_3$ in $E(F)$ such that 
        $[2]Q_3=Q_2$ and by Lemma \ref{Knapp}, $x+u^4$ is a square in $F$, i.e., 
                  \[  
                   x+u^4=\pm \sqrt{-u^4}\sqrt{(-u^4+v^4)}=\pm u^2\sqrt{u^4-v^4} 
                   \]
        is a square in $F$. Since $u^2$ and $-u^2=(iu)^2$ are both  squares in $F$,  $\sqrt{u^4-v^4}$ is also
        a square in $F$. By Lemma 
        \ref{square}, we see that $u^4-v^4$ or $v^4-u^4$ has to be a square in $K$. 
        Therefore, there exist a $t \in K$ such that $(u,v,t)$ or $(v,u,t)$ satisfy the equation $x^4-
        y^4=z^2$ and by Lemma \ref{diophantine}(\ref{difference}), $uvt=0$.
       However, $u$ or $v$ can not be zero since $E$ is non-singular, hence $t$ must be zero which 
       means $a$ equals to $b$; also contradicts to $E$ being nonsingular. 
       Hence this shows that $E(F)$ does not contain a subgroup isomorphic to $\eight \oplus 
       \eight $. Theorem \ref{Burt} implies that $E(F)_{2'}={0}$ and the result follows from Theorem 
       \ref{Laska}. 
       
       See \cite[Proposition 4.1]{Fujita2} for a similar result over $\Q$.
  \end{proof}

   \begin{prop}\label{2,6}
      Assume $E(K)_{\text{tors}}\simeq \two \oplus \six$. Then $E(F)_{\text{tors}}$ is isomorphic to $\four 
      \oplus \Z/12\Z$.
   \end{prop}
   \begin{proof}
       By Theorem \ref{Burt}, all (non-trivial) quadratic twists of $E$ have torsion subgroup isomorphic to 
       $\two \oplus \two $, hence, the odd part of $E(F)_{\text{tors}}$ must be isomorphic to $\Z/3\Z$.
       By Lemma \ref{Knapp}, we also have that $\four \oplus \four \subset E(F)$ since  $\two \oplus \two 
       \subset E(K)$. Hence $\four \oplus \Z/12\Z \subset E(F)$.

       To prove the statement, we need to show there is no point of order $8$ in $E(F)$. Let $P_2$ be in 
       $E(F)$ such that $[2]P_2=P_1=(0,0)$. We compute the $x$-coordinate of $P_2$ by Lemma 
       \ref{Knapp} as 
                            \[ 
                            x(P_2)=\pm \sqrt{ab}.
                            \] 
       Then,
       $P_2$ is either $(\sqrt{ab},\pm \sqrt{ab}(\sqrt{a}+ \sqrt{b}))$ or $(-\sqrt{ab},\pm \sqrt{ab}(-\sqrt{a}+ 
       \sqrt{b}))$. Suppose that $P_2$ is in $[2]E(F)$.  Then again by Lemma \ref{Knapp}, $\sqrt{ab}$ 
       has 
       to be a square in $F$ and by Lemma \ref{square}, $ab$ or $-ab$ is a square in $K$. Suppose that 
       $ab$ is a square in $K$. Then $a=da'^2$ and $b=db'^2$ where $(a',b')$ is a unit.
       Then 
                  \[
                  P_2=(\epsilon_1 da'b', \epsilon_2 da'b'(\epsilon_1 a'\sqrt{d}+b'\sqrt{d}))
                  \] 
        where $\epsilon_1^2=
                         \epsilon_2^2=1$
        and the point 
                  \[
                  (\epsilon_1 da'b', \epsilon_2 da'b'(\epsilon_1 a'+b'))
                  \] 
       defines a point of order $4$ in $E^{(d)}(K)$. However by Theorem \ref{Burt}, we know that any 
       quadratic twist of $E$ has torsion subgroup isomorphic to $\two \oplus \two$ over $K$, hence $ab
       $ is not a square. 
       
       Now assume $-ab$ is a square.  Let $a=da'^2$ and $b=-db'^2$. If $P_2$ is in $[2]E(F)$, then 
       $x(P_2)+a$ is also a square in $F$.       
                \[ 
                x(P_2)+a=\pm \sqrt{ab}+a=\pm (da'b')i+da'^2 =da'(a'\pm b'i)     
                \] 
        Hence $a'+b'i=u^2s$ for some $u$ in $K(i)$ and $s \in K$. We also see that $a'-b'i={\bar{u}}^2s$ 
        where $\bar{u}$ denotes the Galois conjugate of $u$. Then             
                 \[   
                  (a-b)/d= a'^2+b'^2=(u\bar{u})^2s^2 
                  \]
        is a square in $K$. 
        Consider the curve
                 \[ 
                 E'=E(a-b,-b) : y^2=x(x+a-b)(x-b). 
                 \]
        Taking the quadratic twist of $E'=E(a-b,-b)$ 
        by $d$, we obtain               
                  \[ 
                  E'^{(d)} :   y^2=x(x+d(a-b))(x-db). 
                  \] 
        Notice that $d(a-b)=d^2(a'^2+b'^2)$ and $-db=d^2b'^2$ are squares in 
        $K$. Hence $E'^{(d)}$ has a point of order $4$ by Lemma \ref{Knapp}. However, this is not 
        possible by Theorem \ref{Burt} since $E$ and $E'$ are isomorphic over $K$.
        Hence $P_2$ is not in $[2]E(F)$.

        Using the isomorphism between $E(-a,b-a), E: y^2=x(x+a)(x+b)$, and $E(a,b)$ we described 
        earlier, one can show that there is no point $P$ of order $8$ in $E(F)$ which proves that $E(F)\simeq \four \oplus \four 
        \oplus \Z/3\Z$.
        
        See \cite[Proposition 4.3]{Fujita2} for a similar result over $\Q$.
    \end{proof}

   \begin{prop}\label{4,4}
         Let $K=\Q(i)$. If $E(K)_{\text{tors}}\simeq \four \oplus \four$, then $E(F)_{\text{tors}}$ is 
         isomorphic 
      to $\eight \oplus \eight$.

   \end{prop}
   
   \begin{proof}
      Suppose that $\four \oplus \four \subset E(K)$. By Theorem \ref{Ono}, we may assume that 
      $a=s^2$ and $b=t^2$ 
      for some $s,t \in K$. Let 
                            \[ 
                            P_1=(0,0)  \hspace{2mm} \text{and} \hspace{2mm} Q_1=(-s^2,0)    
                            \]
       as before and            
              \[  
              x(P_2)=\pm st   \hspace{2mm} \text{and} \hspace{2mm} x(Q_2)=\pm s\sqrt{s^2-t^2}-s^2 
              \]
       such that $[2]P_2=(0,0)$ and  $[2]Q_2=Q_1$.
     
       Since $Q_2$ has order $4$, it must be in 
       $E(K)$ which forces $s^2- t^2$ to be a square in $K$. Let $r$ be in $K$ such that                      
                                       \[ 
                                       s^2-t^2=r^2. 
                                       \]  
     Hence, we compute that $Q_2$ equals to 
     \[ 
     (sr-s^2,\pm isr(r-s)) \hspace{1mm} \text{or} \hspace{1mm} (-sr-s^2,\pm isr(r+s)).
     \] 
     Let $P_2$ denote the point with $x(P_2)=st$ and let $Q_2$ denote the point with $x(Q_2)= (sr-s^2,\pm isr(r-s))$.    
     By lemma 9, we know that $\eight \oplus \eight 
     \subset E(F)$. We want to show that $\sixteen \not\subset E(F)$.
      Using Lemma \ref{Knapp}, we 
     find a point $P_3$ such that $[2]P_3=P_2$ and          
                       \[ 
                       x(P_3)=1/2\sqrt{st}(\sqrt{s}+\sqrt{t}+\sqrt{s+t})^2. 
                       \]
     Assume $P_3$ is in $[2]E(F)$. Then $st$ is a square in $K$ by Lemma \ref{Knapp} and Lemma 
     \ref{square}. Since $s$ and $t$ are relatively prime, either $s$ and $t$ are both squares or they are 
     both $i$ times a square. Now, let $u,v$ be in $K$ such that $s=iu^2,t=iv^2$.
     Then the equation $s^2-t^2=r^2$ gives us $-u^4+v^4=r^2$ which has no non-trivial solutions over 
     $K$ by Lemma \ref{difference}. Similarly, if $s,t$ are both squares, then $u^4-v^4=r^2$ which proves 
     that $P_3 \not\in [2]E(F)$.  
     
     Similarly using Lemma \ref{Knapp}, we can find another point $Q_3$ where $[2]Q_3=Q_2$ such 
     that
        \begin{align*}
             x(Q_3)+s^2 &=\sqrt{(sr-s^2)sr}+\sqrt{(sr-s^2)(sr-r^2)}+
            \sqrt{(sr-r^2)sr} +sr  \\
                &=s\sqrt{r}\sqrt{r-s}+\sqrt{sr}\sqrt{-(r-s)^2}+r\sqrt{s}\sqrt{s-r}+sr \\
                & =-(-s)\sqrt{r}\sqrt{r-s}+\sqrt{r}\sqrt{-s}(r-s)+r\sqrt{-s}\sqrt{r-s}-(-s)r \\
                &=1/2 \sqrt{-sr}(\sqrt{r}+\sqrt{r-s}-\sqrt{-s})^2 
        \end{align*} 
     If $Q_3$ is in $[2]E(F)$, then $x(Q_3)+s^2$ is a square in $F$. By lemma \ref{square}, $-sr$ is a square in 
     $K$. Hence as we discussed earlier, $s$ and $r$ are either both squares or both $i$ times a 
     square. In both cases, we 
     obtain a non-trivial solution for the equation $x^4-y^4=z^2$ which is not possible. Therefore, $Q_3$ 
     is not in $[2]E(F)$. 
     
     We will show next that there are no points of order $16$ in $E(F)$.      
     Let $R_3$ be in $E(F)$ such that $[2]R_3=P_2+Q_2=R_2$, where $[2]R_2=(-t^2,0)$. We find that
                      \[
                      x(R_3)+t^2=(1/2)\sqrt{tri}(\sqrt{t i}+\sqrt{r}+\sqrt{r+ti})^2. 
                      \]     
     Hence if $R_3$ is in $[2]E(F)$, then $t r i$ has to be a square. This leads to a contradiction to 
     Lemma \ref{diophantine} as earlier.
       Hence $R_3$ is not in $[2]E(F)$ either.          
     
     Notice that $R_3=[a] P_3+ [b]Q_3$ for some odd integers $a,b$. We will assume for 
     simplicity that $a=b=1$. (The following discussion can be modified for general $a,b$.)
     
     Assume that there is a point $P$ of order $16$ in $E(F)$. Then $2[P]=[k]P_3+[l]R_3$ for some 
     $k,l\in \Z$. Define
        \[ 
        Q= \begin{cases} 
           [(k-1)/2]P_3+[l/2]R_3   & \text{if k is odd, l is even} \\
           [k/2]P_3+[(l-1)/2]R_3  & \text{if l is odd, k is even} \\
           [(k+1)/2]P_3+[(l-1)/2]R_3 & \text{if k,l are both odd }
              \end{cases}
        \]
      Then, $[2](P-Q)$ is either $P_3, R_3$ or $Q_3$ which is not possible as we showed earlier. Hence, $\sixteen \not\subset E(F)$ and $E(F)\simeq \eight \oplus \eight $.

    \end{proof}
    
    \begin{prop} \label{2,4}
       Assume that $E(K)_{\text{tors}}\simeq \two \oplus \four$. 
        \begin{enumerate} 
            \item Let $K=\Q(i)$. Then $E(F)$ does not contain a subgroup isomorphic to $\eight \oplus 
                 \eight$.
            \item Let $K=\Q(\sqrt{-3})$. Then $E(F)$ contains a subgroup isomorphic to $\eight \oplus 
                 \eight$ only if $E^{-1}(K)$ has a point of order $4$ in which case $E(F)\simeq \Z/8\Z \oplus 
                      \Z/8\Z$. 
        \end{enumerate}
    \end{prop}

    \begin{proof}
       By Lemma \ref{Knapp}, we know that $\four \oplus \eight \subset E(F)$. We first determine when 
       $E(F)$ has a subgroup isomorphic to $\eight \oplus \eight$. By Theorem \ref{Ono}, we may 
       assume that $a=s^2$ and $b=t^2$ for some $s,t \in \mathcal{O}_K$ relatively prime. Suppose that    
       $\eight \oplus \eight \subset E(F)$. 

       Let  $Q_1=(-s^2,0)$ and let $Q_2$ be in $E(F)$ so that $[2]Q_2=Q_1$.
       We compute                   
                     \[ 
                      x(Q_2)=-s^2+s\sqrt{s^2-t^2}.   
                        \]
         Then there is a point $Q_3 \in E(F)$ such that $[2]Q_3=Q_2$. Hence by Lemma 
         \ref{Knapp}, 
                     \[   
                     x(Q_2)+s^2=s\sqrt{s^2-t^2} 
                     \]
         is a square in $F$ which implies that $\sqrt{s^2-t^2}$ is a square in 
         $F$. Hence either $s^2-t^2$ or $t^2-s^2$ is a square in $K$  by Lemma \ref{square}. 
        
       Let $K=\Q(i)$ and let $s^2-t^2=r^2$ for some $r \in K$. Then we compute that 
                 \[
                 Q_2=(s r-s^2, i s r(r-s))
                  \]  
       The points $P_2$ and $Q_2$ generate $E[4]$ and they are both in $E(K)$. This 
       contradicts the fact that $E(K)\simeq \two \oplus \four$. Note that there is no 
       need to consider the case $t^2-s^2$ is a square separately since $-1$ is a square in $K$.     
       
       Now we may assume that $K=\Q(\sqrt{-3})$. If  $s^2-t^2=r^2$, then $Q_2=(-s^2+sr,isr(r-s))$ in $E(K(i))
       $ induces a point of order $4$ on the quadratic twist $E^{(-1)}$ of $E$, namely $(-s^2+sr,sr(r-s))
       $. 
       Similarly, if $s^2-t^2=-r^2$, then $R_2=(-t^2+rt, irt(r-t))
       $ gives rise to a $K$-point of order $4$ on $E^{(-1)}$. Hence we conclude that $E(F)
       $ does not contain the full $8$-torsion if $E^{(-1)}(K)$ does not have a point of order $4$.
       
       Assume $E^{(-1)}(K)$ has a point of order $4$. Then we can show $E(F)$ contains a subgroup 
       isomorphic to $\eight \oplus \eight$ by computing points of order $8$ as we did in the proof of 
       Proposition \ref{4,4}.
       
       Next we will show that there is no point of order $16$ in $E(F)$. 
       We know that there is a point $P_3$ in $E(F)$ such that $[2]P_3=P_2$. We will show that $P_3 
       \not\in [2]E(F)$. We compute
              \[
              x(P_3)=(1/2) \sqrt{st}(\sqrt{s} +\sqrt{t}+\sqrt{ s+t})^2. 
              \]
        Assume $P_3$ is in $[2]E(F)$. Then $\sqrt{ st}$ is a square in F and so $st$ or $-st$ is a square 
        in $K$. Since $s$ and $t$ are relatively prime, $s=\pm du^2$ and $t= dv^2$ for some unit $d$ 
         in $\mathcal{O}_K$. 
        The only square-free units in $\mathcal{O}_K$ are $\{\pm1\}$, hence $d=\pm1$.
        In each case we obtain a non-trivial solution (over $K$) to the equation $x^4- y^4=z^2$ since 
        $s^2-t^2$ or $t^2-s^2$ is a square.  This is not possible by 
        Lemma \ref{difference}. This shows that $P_3$ is not in $[2]E(F)$. 
        
        Let $s^2-t^2=r^2$. Then using the computations of points $Q_3$ and $R_3$ in the proof of 
        Proposition \ref{4,4}, we see that $Q_3$ and $R_3$ can not be in $[2]E(F)$. This can be shown 
        with a similar argument we used to show $P_3$ is not in $[2]E(F)$. With a similar discussion 
        to the proof of Proposition \ref{4,4} we construct a point $Q$ in $E(F)$ and show that $E(F)$ 
        does not contain a point of order $16$.
       
        If $s^2-t^2=-r^2$, for a similar argument, use $x(Q_3)$ and $x(R_3)$ which we computed in 
        Proposition \ref{4,4}. 
        
        See \cite[Proposition 4.6]{Fujita2} for a similar result over $\Q$.

     \end{proof}
    \begin{theorem}\label{2,2}
    
     Let $K$ be the quadratic field $\Q(i)$ or $\Q(\sqrt{-3})$. Assume 
    $E(K)_{\text{tors}}\simeq \two \oplus \two$. Then $E(F)_{\text{tors}}$ is isomorphic to one of the 
    groups listed in Proposition \ref{2,8}, Proposition \ref{2,6}, Proposition \ref{2,4}, Proposition 
    \ref{4,4}, or the group $\four \oplus \four$.
    \end{theorem}
    
    \begin{proof}
    A quadratic twist of $E$ can have torsion subgroup isomorphic to $\two  \oplus\eight$,  $\two \oplus
    \six$, $\two \oplus \four$, $\two \oplus \two$ or $\four \oplus \four$. 
   Note that since $E$ and $E^{(d)}$ are isomorphic over a quadratic extension of $K$,
          \[ 
          E(F)_{\text{tors}}\simeq E^{(d)}(F)_{\text{tors}}.    
          \]
     Hence if $E^{(d)}(K)_{\text{tors}}\not\simeq\two\oplus \two$ for some $d \in K$, then 
    $E(F)_{\text{tors}}$ will 
     be one of the groups listed in Proposition \ref{2,8}, Proposition \ref{2,6}, Proposition \ref{2,4}, 
     Proposition \ref{4,4}. 
   Therefore we may assume that $E^{(d)}(K)\simeq \two \oplus \two$ for all $d\in K$. The rest of the     
   proof is same as the proof of Proposition \ref{2,6}.
   
   See \cite[Proposition 4.5]{Fujita2} for a similar result over $\Q$.

    \end{proof}

\section{$E(K)_{\text{tors}}$ is Cyclic}

Let $E:y^2=f(x)$ be an elliptic curve with $E(K)_{\text{tors}}\simeq \Z/N\Z$. If $N$ is odd, then there 
    is no point of order $2$ in $E(K)$. Since the $2$-torsion points on $E$ are $(\alpha_i,0)$ where $
    \alpha_i$ are the roots of $f$, $E(K)_{\text{tors}}$ being odd implies $f$ is irreducible over $K$. 
    Therefore, $f$ is irreducible over $F$. Then $E(F)_{\text{tors}}$ is also odd and we analyzed this 
    case in \ref{odd section}. Hence, we assume that $N$ is even.

   We will need the following lemma to show that $\eight \oplus \eight$ is not contained in $E(F)$ when 
   $E(K)$ is cyclic.
        \begin{lemma}\label{square i}
             Let $K$ be the field $\Q(\sqrt{-3})$ and let $F$ be the maximal elementary abelian $2$-extension of $K$. Suppose $a \in K$. Then $\sqrt{a i}$ can not be a square in $F$.
       \end{lemma}
          \begin{proof}
              Suppose that $\sqrt{a i}$ is a square in $F$. Then the proof of Lemma \ref{square} shows that 
              $a i$ is a square in $K(i)$. Then 
                            \[ 
                            a i=b^2(1+i)^2 \hspace{2mm} \text{or} \hspace{2mm} a i=b^2(1-i)^2 \hspace{2mm}  
                                  \text{with} \hspace{2mm} b \in K.
                                  \] 
             Then $a=\pm 2b^2$. Hence $\sqrt{a i}$ is equal to $b(1\pm i)$ and we obtain that 
             $(1\pm i)$ is a square in $F$. 
             Now let $\beta=\sqrt{1+i}$. 
             Hence $(\beta^2-1)^2+1=0$ and we see that $\beta$ is a root of the polynomial 
                                 \[ 
                                 f(x)=x^4-2x^2+2. 
                                 \] 
                  We observe that the degree of the splitting field of $f$ is $8$ and that $F$ has to contain 
                  the 
              splitting 
              field of $f$ since it contains $\beta$, hence the Galois group of $f$ over $K$ has to be an 
              elementary abelian $2$-group. Notice that the Galois group of $f$ is a subgroup of $S_4$. 
              Since $S_4$ does not have any 
              elementary abelian 2-subgroup of order $8$, we get a contradiction. Hence neither $\sqrt{1+i}
              $ nor $\sqrt{1-i}$ is a 
              square in $F$ and this proves that $\sqrt{a i}$ can not be a square in $F$ for any $a \in K$.

   \end{proof}

    \begin{prop} \label{cyclic} 
             Let $E$ be an elliptic curve over $K$ and suppose that \\ $E(K)_{\text{tors}} \simeq \Z/2N\Z$ for   
             some integer $N$.
         \begin{enumerate} 
               \item If $K=\Q(i)$, then $\four \oplus \four \not\subset E(F)$. 
               \item If $K=\Q(\sqrt{-3})$, then
                           $\eight \oplus \eight \not\subset E(F)$. 
          \end{enumerate}
     \end{prop}
     
     \begin{proof}
            Suppose $E$ is given by the equation $y^2=f(x)$. Then $f$ has exactly one root in $K$. 
            Without 
            loss of generality, we        
            may assume that $E$ is given by 
                                 \[
                                 y^2=f(x)=x(x-\alpha)(x-\bar{\alpha})
                                 \]
            where $\bar{\alpha}$ 
            denoted the complex conjugate of $\alpha$. 
            We may write 
            $\alpha$ and $\bar{\alpha}$ as $a+b\sqrt{c}$ and $a-b\sqrt{c}$ for some square-free $c$ since 
            they are  
            defined over a quadratic field. Then 
                            \[
                             \alpha-\bar{\alpha}=2b\sqrt{c}. 
                             \]
             If the point $Q_1=(\alpha,0)$ is in $[2]E(F)$, then by 
             Lemma \ref{Knapp}, $\alpha-\bar{\alpha}$ is a square in $F$  and 
             hence either $c$ or $-c$ is a square by Lemma \ref{square}.     
             If $K=\Q(i)$, then $Q_1 \not\in [2]E(F)$ which proves that $\four \oplus \four \not
             \subset E(F)$ when $K=\Q(i)$.

            Let $K=\Q(\sqrt{-3})$. If $\four \oplus \four \subset E(F)$, then $c=-1$ and we may assume that  
            $\alpha=a+bi$. Let $P_2$ be a point such that $[2]P_2=(0,0)$. 
            Suppose that $E[8] \subset E(F)$. Then $P_2$ is in $[2]E(F)$ and Theorem \ref{Ono} implies 
            that 
            $x(P_2)= \sqrt{\alpha \bar{\alpha}},\hspace{1mm} x(P_2)-\alpha$ and $x(P_2)-\bar{\alpha}$ are 
            all squares in $F$. 
            Hence by Lemma \ref{square},                      
                             \[
                              \alpha\bar{\alpha}=a^2+b^2= d^2 \hspace{2mm} \text{or} \hspace{2mm} \alpha
                                       \bar{\alpha}=a^2+b^2= -d^2 
                                       \]  
            for some $d$ in $K$. We obtain either  
                \begin{equation} \label{square1}
                 (x(P_2)-\alpha)(x(P_2)-\bar{\alpha})=2d(d-a)= e^2   
                \end{equation} 
                or
               \begin{equation} \label{square2} 
               (x(P_2)-\alpha)(x(P_2)-\bar{\alpha})=2d(d+ai)= f^2
               \end{equation}
            for some $e,f \in K$.
           We can parametrize $a,b,d$ as              
                 \[   
                 a=k(m^2-n^2), \hspace{2mm} b=2kmn,\hspace{2mm} \text{and} \hspace{2mm}  
                             d=k(m^2+n^2) 
                             \]
            for some $k,m,n \in \mathcal{O}_K$. 
            We set $d= k i(m^2+n^2)$ if $a^2+b^2=-d^2$. 
           
            Equation \ref{square1} and \ref{square2} gives us either $m^2+n^2$ or  
           $2(m^2+n^2)$ is a square in $K$. Suppose $2(m^2+n^2)$ is a square in $K$, then $m^2+n^2$ 
           is divisible by $2$ and so is $ m^2-n^2$. This means that $a,b,d$ are all divisible by $2$.    
           (Notice that $2$ remains as a prime in $\mathcal{O}_K$.) In this case, we can replace $E$ by 
            the quadratic twist $E^{(2)}$ of $E$ since $E(F)\simeq E^{(2)}(F)$ and they both have 
           cyclic torsion subgroup over $K$. 
           
           Therefore it is enough to consider the case where $a,b$ are not both divisible by $2$.
           We will assume that $m^2+n^2$ is a square in $K$, then we compute $x(Q_2)$ where $Q_1=(\alpha,0)=[2]Q_2$. We know that                   
                            \begin{align*} 
                            x(Q_2)-\alpha &=\sqrt{\alpha(\alpha-\bar{\alpha})}  \\ 
                                                                                &=2(m-n i)\sqrt{mn i} 
                            \end{align*}
  
          is a square in $F$, hence $\sqrt{m n i}$ has to be a square in $F$ but this is not 
          possible by 
           Lemma 
           \ref{square i}. Note that switching the parametrization 
            of $a$ and $b$ does not change the result.  
    \end{proof}

    \begin{prop}[{\cite[Lemma~13]{Fujita2}}]  \label{4} 
                     Let $E(K)$ be cyclic. Then $E(F)$ contains a point of order $4$ if and only if there exist a 
                     $d \in \mathcal{O} _K$ such that $E^{(d)}(K)$ has a point of order $4$.     
     \end{prop}

    \begin{proof} 
    The proof follows the proof of the statement when $K=\Q$ and it is given in         
                          \cite[Lemma~13]{Fujita2}.
       \end{proof}

\section{More Restrictions on the Torsion Subgroups}\label{restrictions2}

\begin{prop}\label{4+20}
                        Let $E$ be an elliptic curve over $K$. Then $E(F)$ cannot  
                                     have a subgroup isomorphic to $\four \oplus \four \oplus \Z/5\Z$.
      \end{prop}

     \begin{proof}
           Suppose that $\four \oplus \four \oplus \Z/5\Z \subset E(F)$. Assume that $E(K)$ is cyclic. Then 
           by Proposition \ref{4}, $E^{(d)}(K)$ has a point of order $4$ for some $d\in K$. Since $E$ and 
           $E^{(d)}$ 
           are quadratic twists, $E(F) \simeq E^{(d)}(F)$ and hence $E^{(d)}(F)$ has a Galois invariant 
           subgroup 
           of order $20$ which is not possible by Proposition \ref{modular curves}. Similarly, the result 
           holds in the case where $E(K)$ contains the full $2$-torsion by our results in \S\ref{full 
           torsion}.
     \end{proof}

      \begin{prop}\label{12} 
                      Let $E$ be an elliptic curve over $K$. Then $E(F)$ cannot have a subgroup 
                       isomorphic to $\Z/12\Z \oplus \Z/12\Z$.
      \end{prop}
      
       \begin{proof}
             If $E(K)$ contains $\two \oplus \two$, then $E(F)$ cannot have a subgroup isomorphic to $\Z/
             3\Z \oplus \Z/3\Z$ by the results of \S\ref{full torsion}. Hence we may assume that $E(K)
             $ is 
             cyclic.
             
             If $K=\Q(i)$, then $E[4]$ is not contained in 
             $E(F)$ by Lemma \ref{cyclic}. 
             
             Let $K=\Q(\sqrt{-3})$. We may assume that $E$ has a point $P$ of order $4$ by Lemma \ref{4} (replacing $E$ by a twist if necessary) and a $K$-rational subgroup $C$ of order $3$ by Lemma \ref{Laska-Lorenz}. Let 
             $\phi: E \to E':=E/C$, then $E'$ has 
             a cyclic isogeny of order $9$ defined over $K$ by \cite[Lemma 7]{Najman-cubic} since $E$ 
             has an additional $K$-rational $3$-cycle. The image of $P$, $\phi(P)$ is in $E'(K)$ and it is 
             of order $4$ since the order of $\ker(\phi)$ is relatively prime to $4$. Then $E'$ has a 
             $K$-rational subgroup of order $36$ which is not possible by Proposition \ref{modular curves}. 
             Hence, there is no such curve over the fields $K$.
       \end{proof}

       \begin{prop} \label{4,32}
                Let $E$ be an elliptic curve over $K$. Then $E(F)$ cannot have a subgroup isomorphic to $  
                \four \oplus \thirtytwo$. 
       \end{prop}

       \begin{proof}

               Suppose that $E$ is an elliptic curve defined over $K$ with $\four \oplus \thirtytwo \subset    
               E(F)$. Then Proposition \ref{2,6}, Proposition \ref{2,2}, and Lemma \ref{4} implies that we 
               may assume that $E(K)$ has a point of order $4$. Let $P_2$ denote such a point in $E(K)$. 
               
               We pick generators  $x,y$ for the $2$-adic Tate module $T_2(E)$ 
               of $E$ such that 
                                         \[ 
                                         x \equiv P_2  \pmod{4}. 
                                         \] 
               Notice that $E[4]\subset E(F)$.
               Then the $2$-adic representation of the group $\Gal(\bar{F}/F)$ is given as follows:
                            \[ 
                            \Gal(\bar{F}/F) \rightarrow \text{Aut}(T_2(E)) 
                             \]
                            \begin{equation}      \label{E}
                                        \rho_2 : \sigma \mapsto  
                                           \left( \begin{array}{cc}
                                                    1+4a_{\sigma} & 4c_{\sigma}  \\
                                                      4b_{\sigma} & 1+4d_{\sigma}  \\
                                                   \end{array} \right) 
                                     \end{equation} 
                  for $\as, \bs, \cs$ and $\ds$ in $\Z_2$.
              Note that $F$ contains a primitive $8$th root of unity. Hence
                             \[ 
                             \det(\rho_2(\sigma)) \equiv 1 \pmod{8} 
                              \] 
                since $\zeta_8^{\det(\rho_2(\sigma))}=\sigma(\zeta_8)=\zeta_8$.
                Computing the determinant, we obtain that $a_{\sigma}+d_{\sigma} \equiv 0 \pmod{2}$. Let       
                $E'=E/{\langle P_1 \rangle}$ where $P_1=[2]P_2$ and $\phi$ be the morphism  $E  
                \rightarrow E'$. We will choose the generators of $T_2(E')$ as $x'$ and $y'$ where $2x'=
                \phi(x)$ and $y'=\phi(y)$. Hence, we find the $2$-adic representation of $\Gal(\bar{F}/F)$ on  
                $T_2(E')$ as:
                          \begin{equation} \label{E'}
                                  \rho_2': \sigma \mapsto 
                                    \left( \begin{array}{cc}
                                            1+4a_{\sigma} & 8c_{\sigma}  \\
                                            2b_{\sigma} & 1+4d_{\sigma}  \\
                                    \end{array} \right)  
                           \end{equation}
         Since $E(K)$ has a point of order $4$, $E'(K)$ contains full $2$-torsion. (see 
          \cite[Example 4.5]{Silverman}). Hence by Lemma \ref{Knapp}, the full $4$-torsion $E'[4]$ is 
          contained in $E'(F)$. Hence the representation 
          in (\ref{E'}) tells us that $\bs$ is divisible by $2$ for every $\sigma \in \Gal(\bar{F}/F)$. Let $\bs=2 
          b_{\sigma}'$. 
          
          Also notice that $E'(F)$ must have a point of order $16$ since $E(F)$ has a point of order $32$. 
          Let $kx'+ly' \pmod{16}$ be such a point for some $k,l \in \Z$ and at least one of $k,l$ is not 
          divisible by $2$. Since this point is in $E(F)$, it is fixed under the action of $\Gal(\bar{F}/F)$. 
           Then we obtain from the representation in (\ref{E'}) that
                   \begin{align}  \label{16}  
                           (1+4\as)k+8\cs l \equiv k \pmod{16} \\
                       \label{16'}     (4\bs')k+(1+4\ds)l \equiv l \pmod{16} 
                     \end{align}
 
            Assume $a_{\sigma}$ is a unit in $\Z_2$ for some $\sigma$, then so is $d_{\sigma}$ since $\as 
            + \ds \equiv 0 \pmod{2}$. An easy computation shows that $k$ and $l$ are both congruent to $0$ modulo $2$ which is a contradiction. 
 
          Hence $\as \equiv 0 \pmod{2}$ for all $\sigma \in \Gal(\bar{F}/F)$ and so is $\ds$. Proposition 
          \ref{cyclic} together with the results of  \S\ref{full torsion} imply that either $E(K)\simeq \four 
          \oplus \four$ or $E(F)$ does not contain $E[8]$. If $E(K)\simeq \four \oplus \four$, then we 
          showed in Proposition \ref{4,4} that it can not have a point of order $16$.
 
          Hence, we may assume that $E(F)$ does not contain $E[8]$ and it implies that $b'_{\sigma}$ is 
          not divisible by $2$ for some $\sigma_1$. See the representation in (\ref{E'}). 
          
          Once again using (\ref{16}) and (\ref{16'}), we compute that
          $\cs \equiv 0 \pmod{2}$ for all $\sigma$. Then the representation of $T_2(E)$ in (\ref{E}) implies 
          that $E[8]$ is contained in $E(F)$ and we get a contradiction.
          See \cite{Fujita2} for the case $K=\Q$.
    \end{proof}

      \begin{corollary}
             Let $E: y^2=x(x+a)(x+b)$ be an elliptic curve defined over $K$. Assume that         
             $E(K)_{\text{tors}}\simeq \two \oplus \eight$. Then 
                            \[ E(F)_{\text{tors}} \simeq \four \oplus \sixteen. \]
      \end{corollary}

   \begin{proof} The statement follows from the Proposition \ref{2,8} and Proposition \ref{4,32}.
   \end{proof}

   \begin{prop}\label{8,12}
         Let $E/K$ be an elliptic curve. Then $E(F)$ cannot be isomorphic to $\four \oplus \eight \oplus \Z/ 
         3\Z$.
   \end{prop}

  \begin{proof}
       We will proceed the same way as in Proposition \ref{4,32}. The representation of $\Gal(\bar{F}/F)$     
       on $T_2(E)$ and $T_2(E')$ is same as given in the proof of  Proposition \ref{4,32}.
       Since $E'(F)$ contains full $2$ torsion and also a point of order $3$, $E'(F) \simeq \four \oplus \Z/ 
       12\Z$ by Proposition \ref{2,6}. Hence $\bs \equiv 0 \pmod{2}$ for all $\sigma$. If $\as \equiv 0  
       \pmod{2}$ for all $\sigma$, then so is $\ds$ and $y' \pmod{8}$ is stabilized under the action of $
       \Gal(\bar{F}/F)$. However, $E'(F)$ does not have a point of order $8$. Hence $\as$ is not divisible 
       by $2$ for some $\sigma_1$. Then similar to Proposition \ref{4,32}, we obtain congruences as in (\ref{16}) and 
       (\ref{16'}) (replacing modulo $16$ by $8$), we get a contradiction.
   \end{proof}

\section{Main Result}

 \begin{theorem} \label{main}
       Let $K$ be a quadratic cyclotomic field, let $E$ be an elliptic curve over $K$, and let $F$ be the 
       maximal elementary abelian $2$-extensions of $K$.
              \begin{enumerate}
                  \item If $K=\Q(i)$, then
                        $E(F)_{\text{tors}}$ is isomorphic to one of the following groups:
                       \begin{align*}
                          & \Z/{2}\Z \oplus \Z/{2N}\Z & &    ( N=2,3,4,5,6,8) & &\\
                          & \Z/{4}\Z \oplus \Z/{4N}\Z & &    ( N=2,3,4) & &\\
                          &\Z/N\Z \oplus \Z/N\Z && (N=2,3,4,6,8) & & 
                       \end{align*}
                 \noindent    or $\{1\}, \Z/3\Z$, $ \Z/5\Z$, $ \Z/7\Z$, $ \Z/9\Z$, $ \Z/15\Z$.                                  
                  \item If $K=\Q(\sqrt{-3})$, then $E(F)$ is either isomorphic to one of the groups listed above 
                  or 
                              \[ 
                              \Z/2\Z \oplus \Z/32\Z. 
                              \]

              \end{enumerate}
          \end{theorem}
\begin{proof}

    We begin with the list given in Theorem \ref{Laska}. Suppose that $E(F)_{(2)}\neq {1}$. If 
      $E(F)_{2'}={1}$, then either $\eight \oplus \eight$ is not contained in $E(F)$ or $E(F)\simeq \eight 
      \oplus \eight$ by our results in \S\ref{full torsion} and Proposition \ref{cyclic}. Then with the 
       notation of Theorem \ref{Laska}, if $b=3$, then $r=0$. By Proposition \ref{4,32}, if $b=2$, then $r
      \leq 2$. We obtain the groups:
                        \begin{align*}  
                                            & \Z/{2}\Z \oplus \Z/2N\Z & \hspace{2mm} \text{for} \hspace{2mm}    &  
                                               N=1,2,4,8,  \\
                                            & \four \oplus \Z/4N\Z &   \hspace{2mm} \text{for} \hspace{2mm}    &    
                                                     N=1,2,4,    \\
                                             &  \eight \oplus   \eight   & &   \end{align*} 
             and $ \Z/2\Z \oplus \Z/32\Z $ if $K=\Q(\sqrt{-3})$.         

      Suppose $E(F)_{2'}\simeq \Z/3\Z$. If $\two \oplus \two$ is contained in $E(K)$, then $E(K) \simeq 
      \two \oplus \six$ and we showed that 
      $E(F) \simeq \four \oplus \Z/12\Z$. Otherwise, we know that $E(F)$ cannot contain $\four \oplus       
      \four$ if $K=\Q(i)$. Similarly $E(F)$ cannot contain $\eight \oplus \eight$ if $K=\Q(\sqrt{-3})$. Along     
      with Proposition \ref{8,12}, we are left with three possible groups:                   
              \[ 
              \two \oplus \six, \hspace{2mm} \two \oplus \Z/12\Z, \hspace{2mm} \four \oplus \Z/12\Z
               \]
      The case $E(F)_{2'} \simeq \Z/3\Z \oplus \Z/3\Z$ follows from Proposition \ref{12}, Proposition 
      \ref{cyclic} and Theorem \ref{Laska}. The only possible group is 
                          \[  
                          \Z/6\Z \oplus \Z/6\Z.  
                            \] 
      Similarly 
      when $E(F)_{2'} \simeq \Z/5\Z$, it follows from Proposition \ref{cyclic}, Lemma \ref{4+20} and  
      Theorem \ref{Laska} that the only option is
                   \[     
                   \Z/2\Z \oplus \Z/10\Z.   
                    \] 
      The case where 
                 $E(F)_{(2)} ={1}$ was studied in the first section and we found the groups 
                                \[
                                 \{1\}, \Z/3\Z,  \Z/5\Z, \Z/7\Z, \Z/9\Z, \Z/15\Z  \hspace{1mm} \text{and} 
                                         \hspace{1mm} \Z/3\Z \oplus \Z/3\Z . 
                                \]
  \end{proof}                                             
        
        \begin{remark}
              Every group we listed in Theorem \ref{main} except $\Z/2\Z \oplus \Z/32\Z$ also appears as   
              the torsion subgroup of some elliptic curve defined over $\Q$ in its maximal elementary 
              abelian 2 extensions. We were able to prove neither the nonexistence of an elliptic curve 
              defined over $\Q(\sqrt{-3})$ with such a subgroup in $E(F)$ nor give an example of such a  
              curve. 
        \end{remark}

   \begin{acknowledgements}

        The author wishes to thank Sheldon Kamienny for suggesting this problem and for his kind 
        support. Samir Siksek provided valuable insight for a part of the proof of Theorem \ref{21}. This 
        work also greatly benefited from conversations with Burton Newman and Jennifer Balakrishnan. 
        We also thank the anonymous referee for the remarks on the previous draft.

\end{acknowledgements}


\end{document}